\documentclass[10pt,reqno]{amsart}

\usepackage[all,cmtip]{xy}
\usepackage{amsmath,amssymb}
\usepackage{amscd,latexsym}
\usepackage{verbatim}
\usepackage[all]{xy}
\usepackage{tikz}
\usepackage{caption}

 \usepackage[OT2,T1]{fontenc}

\DeclareSymbolFont{cyrletters}{OT2}{wncyr}{m}{n}
\DeclareMathSymbol{\Sha}{\mathalpha}{cyrletters}{"58}

 \numberwithin{equation}{section}
\newcommand{\nc}{\newcommand}
\nc{\nt}{\newtheorem}
\nc{\dmo}{\DeclareMathOperator}
\nc{\enm}{\ensuremath}

\newtheorem{thm}{Theorem}

\newtheorem{prop}{Proposition}
\nt{claim}{Claim}
\newtheorem{lemma}{Lemma}
\nt{defn}{Definition}
\newtheorem{cor}{Corollary}
\nt{assumption}{Assumption}

\dmo{\Ind}{Ind}
\dmo{\cInd}{c-Ind}
\dmo{\Adj}{Ad}
\dmo{\PGL}{PGL}
\dmo{\SO}{SO}
\dmo{\Lie}{Lie}
\dmo{\Alt}{Alt}
\dmo{\reg}{reg}
\dmo{\sing}{sing}
\dmo{\supp}{supp}
\dmo{\tr}{tr}
\dmo{\Sym}{Sym}
\dmo{\Hom}{Hom}
\dmo{\Tor}{Tor}
\dmo{\Out}{Out}
\dmo{\Ht}{ht}
\dmo{\End}{End}
\dmo{\Mat}{Mat}
\dmo{\Tr}{Tr}
\dmo{\Isom}{Isom}
\dmo{\Span}{Span}

\dmo{\SL}{SL} \dmo{\sgn}{sgn} \dmo{\GL}{GL}  \dmo{\Mod}{mod} \dmo{\geo}{geo} \dmo{\re}{Re} \dmo{\Spec}{Spec}
\dmo{\Fr}{Fr} \dmo{\vol}{vol} \dmo{\Sets}{Sets} \dmo{\im}{im} \dmo{\diag}{diag} \dmo{\Ker}{Ker} \dmo{\val}{val} \dmo{\ord}{ord}
\dmo{\Stab}{Stab} \dmo{\Ad}{Ad} \dmo{\rank}{rank} \dmo{\Symp}{Sp} \dmo{\Nm}{Nm} \dmo{\Norm}{Norm}

\nc{\bG}{\enm{{\mathbf G}}}

\nc{\Aff}{\mathbb{A}}

\nc{\eps}{\varepsilon}

\nc{\ups}{\upsilon}

\nc{\bks}{\enm{{\backslash}}}

 \nc{\isom}{\enm{{\overset{~}{\rightarrow}}}}

 \nc{\Z}{\enm{{\mathbb Z}}}

\nc{\Zp}{\enm{{\mathbb Z_p}}}

\nc{\Gm}{\enm{{\mathbb G_m}}}
\nc{\F}{\enm{{\mathbb F}}}
\nc{\Fp}{\enm{{\mathbb F}_p}}
\nc{\Fq}{\enm{{\mathbb F}_q}}
\nc{\Q}{\enm{{\mathbb Q}}}
\nc{\Qp}{\enm{{\mathbb Q_p}}}
\nc{\R}{\enm{{\mathbb R}}}
\nc{\N}{\enm{{\mathbb N}}}
\nc{\C}{\enm{{\mathbb C}}}
\nc{\CC}{\enm{{\mathcal C}}}
\nc{\half}{\enm{{\frac{1}{2}}}}
\nc{\BB}{\enm{{\mathcal B}}}
\nc{\flip}{\tilde{\eps}}

\nc{\ii}{\enm{{\mathcal I}}}
\nc{\jj}{\enm{{\mathcal J}}}
\nc{\OO}{\enm{{\mathcal O}}}
\nc{\f}{\enm{{\mathcal F}}}

\nc{\GGl}{\enm{{\mathfrak gl}}}
\nc{\GG}{\enm{{\mathfrak g}}}
\nc{\gd}{\enm{{\hat{\mathfrak g}}}}
\nc{\gm}{\enm{{\gamma}}}
\nc{\hh}{\enm{{\mathfrak h}}}

\nc{\II}{\enm{{\mathfrak a}}}
\nc{\LL}{\enm{{\mathfrak l}}}
\nc{\mm}{\enm{{\mathfrak m}}}
\nc{\pp}{\enm{{\mathfrak p}}}
\nc{\TT}{\enm{{\mathfrak t}}}
\nc{\Nc}{\enm{{\mathcal N}}}
\nc{\Cc}{\enm{{\mathcal C}}}
\nc{\HH}{\enm{{\mathfrak h}}}
\nc{\LS}{\enm{{\mathfrak s}}}
\nc{\iso}{\tilde{\rightarrow}}
\dmo{\res}{res}
 \nc{\Gd}{\enm{{\hat{G}}}}

  \nc{\Hd}{\enm{{\hat{H}}}}

\nc{\vt}{\enm{\vartheta}}
\nc{\lra}{\enm{\longrightarrow}}
\nc{\ra}{\enm{\rightarrow}}
\nc{\lip}{\enm{\langle}}
\nc{\rip}{\enm{\rangle}}

\nc{\nn}{\enm{\mathfrak n}}

\nc{\bsk}{\bigskip}

 \nc{\ol}{\overline}
\nc{\ul}{\underline}

\nc{\bH}{{\bf H}}
\nc{\bS}{{\bf S}}
\nc{\bT}{{\bf T}}
\nc{\bB}{{\bf B}}
\nc{\bA}{{\bf A}}
\nc{\bU}{{\bf U}}
\nc{\bN}{{\bf N}}
\nc{\bE}{{\bf E}}
\nc{\bX}{{\bf X}}
\nc{\bM}{{\bf M}}
\nc{\bW}{{\bf W}}

\author{Arnab Mitra}
\author{Steven Spallone}

\address{Tata Institute of Fundamental Research, Colaba, Mumbai-400005, India}
\email{00.arnab.mitra@gmail.com}
\address{Indian Institute of Science Education and Research, Pune-411021, India}
\email{sspallone@gmail.com}
\keywords{Integration formula, maximal parabolic, unipotent radical, Langlands-Shahidi method, intertwining operator}
\subjclass{Primary 22E35, Secondary 22E50}
\begin{document}

\title [Unipotent Radicals] {An Integration Formula for Unipotent Radicals}
\maketitle

\begin{abstract}
Let $P=MN$ be a maximal parabolic of a classical group over a field $F$.  Then the Levi subgroup $M$ is isomorphic to the product of a classical group and a general linear group, acting on vector spaces $X$ and $W$, respectively.  In this paper we decompose the unipotent radical $N$ of $P$ under the adjoint action of $M$, assuming $\dim W \leq \dim X$ and that $\dim W$ is even.  When $F$ is a local field, we obtain a Weyl-type integration formula for $N$.

\end{abstract}

\section{Introduction}   
Let $F$ be a field.  Let $V$ be a finite dimensional inner product space over $F$ which is either orthogonal or symplectic.   Let $G^1$ be the isometry group of $V$.

 Let $P=MN$ be a maximal parabolic subgroup of $G^1$.  Here $N$ is the unipotent radical of $P$, and $M$ is a Levi component of $P$.  Then $M$ is isomorphic to a product of groups $G \times H$, where $G=\GL(W)$ for an isotropic subspace $W$ of $V$, and $H$ is the group of isometries of a nondegenerate subspace $X$ of $V$.  Roughly speaking, our goal is to provide a decomposition of $N$, or rather of a dense open subset $N_{\reg}$, under the adjoint action of $M$.   
In this paper we assume that $\dim W \leq \dim X$ and that   $\dim W$ is even.  
 
Goldberg and Shahidi define in \cite{GS98} a map which we write as $\Norm: N' \to H$, with $N'$ a certain open subset of $N$.  This map sends $\Ad(M)$-orbits in $N'$ to conjugacy classes in $H$.  
(Originally Shahidi \cite{S95} introduced this map in the even orthogonal case.)  The image is contained in the subset 
\begin{equation*}
H_k= \{ h \in H \mid \rank(h-1; X) \leq k \}.
\end{equation*}
  
Fix a torus $S$ in $H_k$ of dimension $\frac{k}{2}$.   We refer to such tori as ``maximal $k$-tori''.  We write $S_r$ for suitably regular elements of $S$, and $H^S$ for the conjugates of $S_r$.  We demonstrate that $H^S$ is (Zariski) open in $H_k$.   We write $E=X^S$ for the fixed point space of $S$, and then $E^\circ$ for the orthogonal complement of $E$ in $X$.  Let $H^E$ be the subgroup of $H$ which fixes $E$ pointwise.  By choosing an isomorphism from $E^\circ$ to $W$, we can identify $H^E$ with the fixed points $G^\theta$ of an involution $\theta$. 
 
Given $\gm \in S_r$, there is a matching semisimple element $\gm_G \in G$ with the property that $\gm_G \cdot \theta(\gm_G)=-\gm^{-1}$, where the right hand side is viewed as an element of $G^\theta$.  (Compare \cite{KS}.)  A variant of $\gm_S$ gives an element $n_S(\gm)$ with $\Norm(n_S(\gm))=\gm$.  The reader will find typical $n_S(\gm)$ written out in matrix form in Section \ref{MatrixSection}.

We obtain from this a Zariski dense open subset
\begin{equation} \label{first}
N_{\reg}= \bigcup_S \{ \Ad(m) n_S(\gm) \mid m \in M, \gm \in S_r \}
\end{equation}
of $N$.  The union is over maximal $k$-tori $S$ of $H$.

Next, suppose $F$ is a local field.  For $m \in M$, put $\delta_N(m)=\det(\Ad(m); \Lie(N))$.
Then the decomposition (\ref{first}) of $N_{\reg}$ leads to the measure-space decomposition:

\begin{equation*}
\int_N f(n) dn= \sum_S |W_H(S)|^{-1} \int_S |\delta(\gm)| \int_{M/{\Delta_S}} f(\Ad(m)n_S(\gm)) |\delta_N(m)| \frac{dm}{dz} d \gm.
\end{equation*}
 The sum is over the $H$-conjugacy classes of maximal $k$-tori $S$ in $H$.  The subgroup $\Delta_S$ is the stabilizer in $M$ of $n_S(\gm)$ for $\gm \in S_r$.
 As usual, $W_H(S)$ denotes the Weyl group of $S$ in $H$.  All measures are conveniently normalized Haar measures.  The product $|\delta(\gm) \delta_N(m)|$ is the Jacobian of the map
\begin{equation*} 
\Sha_{S}: M/\Delta_S \times S_{r}\to N
\end{equation*}
given by $\Sha_S(m \times \gm)=\Ad(m)n_S(\gm)$; we compute this Jacobian explicitly in Section \ref{routes}.
 
One can eliminate the factor $|\delta_N(m)|$ from the right-hand side by replacing $dn$ with an $\Ad(M)$-invariant measure $d_mn$ (specified in Section \ref{Ad-invt}).
The result is:
 
\begin{thm} \label{intro_thm}
Suppose that $\dim W \leq \dim X$, and that $\dim W$ is even.  Let $f \in L^1(N,d_mn)$.  Then 
\begin{equation*}
\int_N f(n) d_mn=\sum_{S}|W_H(S)|^{-1}\int_{S}D_{\Sha}(\gm) \int_{M/{\Delta_S}}f(\Ad(m)n_S(\gamma))\frac{dm}{dz}d\gamma,
\end{equation*}

\begin{equation*}
 \text{where } D_{\Sha}(\gm)=|D_{H^E}(\gm)|^\half |D^\theta_G(\gm_G)|^\half |\det(\gm-1;E^\circ)|^{\half \dim E}.
 \end{equation*}
\end{thm}
   
Here $D_{H^E}$ is the usual discriminant, and $D^\theta_G$ is a $\theta$-twisted discriminant. 
 
These formulas generalize formulas from \cite{GS98} and \cite{IFS} in the case of $\dim W=\dim X$, and from \cite{GSIII} in the case $\dim X=\dim W+1$.
For us, there were significant difficulties not found in the equal-size case.  Rather than realizing $H$ as a subgroup of $G$, we needed to realize a suitable {\it subgroup} of $H$ as a subgroup of $G$. Of course, the image of $\Norm$ lies in a subset of $H$ of measure $0$.  But the real difficulty was to properly compare the tangent spaces of
$M/\Delta_S \times S$ and $N$ for the Jacobian calculation.  Only with careful decomposition and eventual luck was this possible.

The theory developed in this paper will be used in an ongoing project to study intertwining operators, which are given by integrating over unipotent radicals.  Here $F$ is $p$-adic.  Briefly, the geometry of the interaction of $G$ and $H$ through the adjoint action on $M$ on $N$ encodes the Langlands functoriality between the representation theory of $G$ and $H$.  These intertwining operators are used to define Langlands-Shahidi $L$-functions, and the project thus connects the theory of $L$-functions to functoriality.  We refer the reader to the papers (\cite{S92},  \cite{S95}, \cite{GS98}, \cite{GS01}, \cite{GSIII}, \cite{Spallone},  \cite{Comp}, \cite{WWL})   of Goldberg, Shahidi, Wen-Wei Li, and the second author for details.

We emphasize that the best results thus far are only in the case of $\dim X=0$ (\cite{S92}) or $\dim W=\dim X$ (\cite{GS98}, \cite{Spallone}, \cite{Comp}, \cite{WWL}), but the results of this paper will open up many interesting cases with 
$\dim W < \dim X$.  The explicit Jacobian calculations at the end of this paper (as in \cite{IFS}) are crucial, because they allow for the application of endoscopic transfer to this project, as in \cite{WWL}.
This paper also fits into the program of Cogdell and Shahidi \cite{Kudla} of computing generalized Bessel functions.  We have decomposed a measure on the space of $\Ad(M)$-orbits of $N$, a homogeneous space, and this measure is an important instance of the generalized functionals considered in this program.

We now delineate the sections of this paper.  After the preliminaries in Section \ref{Preliminaries}, we review the Goldberg-Shahidi Norm map in Section \ref{Norm_Corr}.  Next, in Section \ref{sect-kmaxtori} we begin to build the theory of $k$-maximal tori, and obtain a density result for suitably regular elements.  In Section \ref{Section_Section} we construct our sections $n_S(\gm)$.
The map $\Sha$ is defined next in Section \ref{alg_theory} and we compute its fibre and image.  In Section \ref{Lie} we  carve up the tangent spaces for the domain and range of $\Sha$.  Thus far, everything is valid for an arbitrary field $F$.  
 
Next, we assume that $F$ is a local field.  We study the derivative of $\Sha$ in Section  \ref{local_now}, and calibrate differential forms on all the pieces.  Next in Section \ref{routes} we break up our spaces further according to root space decompositions and find invariant spaces for the differential.  Finally in Section \ref{LastSection} we consolidate our work and deduce the integration formula.
 
The authors would like to thank Sandeep Varma and Vivek Mallick for useful conversations. This paper is a part of the first author's thesis and he would like to thank his advisor Dipendra Prasad for constant encouragement. This work was initiated during the second author's pleasant visit to the Tata Institute of Fundamental Research in Mumbai, and it is a pleasure to thank the institute for its support.

\section{Preliminaries} \label{Preliminaries}
\subsection{Notation}
Throughout this paper $F$ will be a field of characteristic not two.   From Section \ref{local_now} onwards, we will assume that $F$ is a local field.
If $V$ is a vector space over $F$ with a nondegenerate bilinear form $\Phi$ we write
\begin{equation*}
\Isom(V,\Phi)=\{ g\in \GL(V)\ |\ \Phi(gv_{1},gv_{2})=\Phi(v_{1},v_{2})\ \forall v_{1},v_{2}\in V \}
\end{equation*}
for the group of linear isometries of $V$.

If $A$ is a subspace of $V$ we write $A^\circ$ for the set of vectors perpendicular to $A$.  We say that $A$ is nondegenerate if the restriction of $\Phi$ to $A$ is nondegenerate; in this case $A^\circ$ is also nondegenerate.

If $G$ is a group and $S$ is a subgroup, we write $Z_G(S)$ for the centralizer in $G$ of $S$, and $N_G(S)$ for the normalizer in $G$ of $S$.
Our varieties are usually defined over $F$.  We use normal script (i.e. ``$G$''), respectively bold script (i.e. ``$\bG$'') for the $F$-points, respectively for the $\ol F$-points of these varieties.
If $\bG$ is an algebraic group we write $\bG^\circ$ for the identity component of $\bG$ in the Zariski topology.

 \subsection{The Unipotent Radical}
Let $V$ be a finite dimensional $F$-vector space with a nondegenerate bilinear form $\Phi$.  
Assume that $\Phi$ is either symmetric or antisymmetric.  Let $G^1=\Isom(V,\Phi)$. 
 Let $W$ be an even-dimensional totally isotropic subspace of $V$, and $P$ the stabilizer of $W$ in $G^1$.  Then $P$ is a parabolic subgroup of $G^1$.

Pick a subspace $W'$ of $V$ so that $W+W'$ is direct and nondegenerate. Let $X=(W+W')^{\circ}$. 
Let $M$ be the subgroup of $G^1$ that preserves $W$, $X$ and $W'$; it is a Levi subgroup of $P$.   Let $G=\GL(W)$ and $H=\Isom(X,\Phi |_X)$.  Given $g \in G$ and $h \in H$, write $m(g,h)$ for the element in $G^1$ whose restriction to $W$ is $g$ and whose restriction to $X$ is $h$.  Then $m( \cdot ,\cdot )$ is an isomorphism from $G \times H$ to $M$.
 
Let $N$ be the unipotent radical of $P$.  An element $n\in N$ is determined by linear maps
$$\xi:X\to W, \ \xi':W'\to X, \ \eta:W'\to W$$
such that:
\begin{enumerate} 
\item If $w \in W$, then $n(w)=w$.
\item If $x\in X$, then $n(x)=x+\xi(x)$. 
\item If $w'\in W'$, then $n(w')=w'+\xi'(w')+\eta(w')$.
\end{enumerate}
Define $\xi^{*}:W'\to X$ by $$\Phi(\xi^{*}(w'),x)=\Phi(w',\xi(x)),$$ for $x\in X$ and $w'\in W'$. Similarly define $\eta^{*}:W'\to W$, the adjoint of $\eta$, by $$\Phi(\eta^{*}(w_1'),w_2')=\Phi(w_1',\eta(w_2')),$$ for $w_1',w_2'\in W'$.  Since $n$ is determined by $\xi$ and $\eta$ we write $n=n(\xi,\eta)$. 

The condition that $n(\xi,\eta) \in G^1$ is equivalent to the two conditions 
\begin{enumerate}
\item $\xi^{*}+\xi'=0$,
\item $\eta^{*}+\eta=\xi\xi'$. 
\end{enumerate}
Thus $\xi'$ is determined by $\xi$ and we see that $n(\xi,\eta) \in N$ if and only if  
\begin{align}\label{defn}
\eta+\eta^{*}+\xi\xi^{*}=0.
\end{align}
We have $n(\xi_{1},\eta_{1})n(\xi_{2},\eta_{2})=n(\xi_{1}+\xi_{2},\eta_{1}+\eta_{2}-\xi_{1}\xi_{2}^{*})$ and $n(\xi,\eta)^{-1}=n(-\xi,\eta^*)$.  
 
We have $\Ad(m(g,h))(n)=mnm^{-1}=n(g\xi h^{-1}, g\eta g^{*})$.

\section{Norm Correspondence} \label{Norm_Corr}

Write $N'$ for the matrices $n(\xi, \eta) \in N$ with $\eta$ invertible.  It is clearly an open subset of $N$. 

\begin{defn}
Suppose that $n=n(\xi,\eta) \in N'$.  Let $\Norm(n):X\to X$ be the linear transformation given by 
\begin{equation*}
\Norm(n)=1+\xi^{*}\eta^{-1}\xi.
\end{equation*}
\end{defn}

Here are some first properties of $\Norm$.

\begin {lemma}\label{basicnorm}
\begin{enumerate}
\item $\Norm(n)\in H$, and $\rank(\Norm(n)-1; X) = \dim W$. 
\item If $n \in N'$ and $m(g,h) \in M$, then $\Norm(\Ad(m(g,h))n)=\Ad(h)\Norm(n)$.
\item If $g\in G$, then the pair $(g\xi,g\eta g^{*})$ also satisfies (\ref{defn}) and 
$$\Norm (n(\xi,\eta))=\Norm(n(g\xi,g\eta g^{*})).$$ 
\end{enumerate}
\end{lemma}
 
The following is Lemma 4.1 in \cite{S95}.  For completeness, we rewrite the proof in our notation.
\begin {lemma}\label{gentoclass}
Let $g\in \GL(X)$ be such that $\xi\xi^{*}=(\xi g)(\xi g)^{*}$. Then there exists $h\in H$ such that $\xi g=\xi h$.
\end{lemma}

\begin{proof}
Let $U\subset X$ be the image of $\xi^{*}$.  For $w_{1}',w_{2}'\in W'$, one has
$$\Phi(g^{*}\xi^{*}w_{1}',g^{*}\xi^{*}w_{2}')=\Phi(w_{1}',(\xi g)(\xi g)^{*}w_{2}')=\Phi(w_{1}',\xi \xi ^{*}w_{2}')=\Phi(\xi^{*}w_{1}',\xi^{*}w_{2}').$$
Thus we see that $g^{*}|_{U}: U \to X$ is an isometry. By Witt's theorem, $g^{*}|_U$ extends to an element $h^{*}$ of $H$ such that $h^{*}: X \to X$ is an isometry. Therefore, $g^{*}\xi^{*}=h^{*}\xi^{*}$, giving the lemma.
\end{proof}

\begin{cor}\label{cogentoclass}
Suppose $\xi_{1}, \xi_2: X \to W$ are surjections such that $(\xi_{1},\eta)$ and $(\xi_2,\eta)$ both satisfy (\ref{defn}). Then there is an $h\in H$ so that $\Norm(n(\xi_{1},\eta))=h^{-1}\Norm(n(\xi_2,\eta))h$.
\end{cor}

\begin{proof}
Since both $\xi_1$ and $\xi_2$ are of full rank, there exists $g\in \GL(X)$ such that $\xi_{1}=\xi_2 g$. Also since both of them satisfy equation (\ref{defn}) we have
$$\xi_{2}\xi_{2}'=\xi_{1}\xi_{1}'=(\xi_2g)(\xi_2 g)'.$$Now by Lemma \ref{gentoclass}, there exists $h\in H$ such that $\xi_2 g=\xi_2 h$. A simple calculation now shows that $\Norm(n(\xi_{1},\eta))=h^{-1}\Norm(n(\xi_2,\eta))h$.
\end{proof}

\section{Tori in $\bH_k$}\label{sect-kmaxtori}
\subsection{The $k$-regular Set}
The goal of this section is to establish a subset $\bH_{k,r}$ of $\bH$ which is open in the image of $\Norm$. 
We will use boldface notation to denote the $\ol F$-points of the various varieties, so in particular the tori are all split.  Of course, we are considering the Zariski topology.

Throughout, we will fix an even integer $k \leq \dim X$.

\begin{defn}  Put
\begin{equation*}
\bH_k=\{ h\in \bH \mid \rank(h-1; X) \leq k \}.
\end{equation*}
  We are interested in the tori $\bS \subseteq \bH_k$ of the greatest possible dimension, which is $\frac{k}{2}$.  We call such an $\bS$ (defined over $F$) a $k$-maximal torus.
\end{defn}
  
Fix a maximal torus $\bT$ in $\bH$.    Put $\bT_k=\bT \cap \bH_k$.
 
\begin{lemma} \label{bounce}
\begin{enumerate}
\item The irreducible components of $\bT_k$ are the $k$-maximal tori of $\bH$ contained in $\bT$.  In particular, there are finitely many.
\item Any two $k$-maximal tori in $\bH$ are conjugate.
\item The dimension of the fixed point space ${\bf X}^{\bS}$ is $\dim {\bf X}-k$.
\end{enumerate}
\end{lemma}

\begin{proof}  We may assume $\bT$ to be diagonal, with elements of the form $$\gm=\diag(t_1, t_2, \ldots, t_2^{-1}, t_1^{-1}).$$ In this way we identify $\bT$ with $\Gm^r$, where $r$ is the rank of $\bH$.  We then have
\begin{equation*}
\bT_k=\left\{ (a_1, \ldots, a_r) \mid \text{ at least $\frac{k}{2}$ of $a_1, \ldots, a_r$ are equal to $1$} \right\}.
\end{equation*}
A $k$-maximal subtorus $\bS$ of $\bT$ corresponds to a subset $\sigma$ of $\{1, \ldots, r \}$ with $\frac{k}{2}$ elements, via
\begin{equation*}
\bS=\bS_{\sigma}=\{ (a_i) \in \bT \mid a_i=1 \text{ unless } i \in \sigma \}.
\end{equation*}

It is easy to see that the $\bS_{\sigma}$ are conjugate via the Weyl group of $\bT$ in $\bH$, and the last statement is clear.
\end{proof}

Now let $\bS$ be a $k$-maximal torus in $\bH$, and write $S$ for the $F$-points of $\bS$.  Let $\bE=\bX^{\bS}=\{ x \in \bX \mid sx=x \text{ for all } s \in \bS \}$.  We will need the fact that $\bE$ is nondegenerate.
 
\begin{lemma} \label{iamalemma}Let $(X,\Phi)$ be a nondegenerate inner product space, and $\gm$ a semisimple isometry.  Then the fixed point space $X^\gm$ is nondegenerate.
\end{lemma}
\begin{proof}
We may assume $1$ is an eigenvalue of $\gm$.
Let $x \in X^\gm$ be nonzero. There exists $y \in X$ such that $\Phi(x,y)\neq 0$.  Let $m(t)$ be the minimal polynomial of $\gm$. So, $m(t)=(t-1)f(t)$ with $\gcd \big(t-1,f(t)\big)=1$.  Pick $\alpha,\beta \in F[t]$ with
$$\alpha(t)(t -1)+\beta(t)f(t)=1.$$
Define maps $P,Q:X\to X$ by setting $P=\alpha(\gamma)(\gamma -1)$ and $Q=\beta(\gamma)f(\gamma)$.

One checks that $Q(y) \in X^\gm$ and $\Phi(x,Py)=0$.  Thus $\Phi(x,Qy)=\Phi(x,y)\neq 0$ and therefore $X^\gamma$ is nondegenerate.
\end{proof}

\begin{prop}\label{Enondeg}
$\bE$ is a nondegenerate subspace of $\bX$ of codimension $k$.
\end{prop}

\begin{proof}
There is a finite subset $\Gamma$ of $\bS$ so that $\bX^{\bS}=\bX^{\Gamma}$.  Repeated application of Lemma \ref{iamalemma} to the elements of $\Gamma$ shows that $\bX^{\bS}$ is nondegenerate.
The dimension statement was given in Lemma \ref{bounce}.
\end{proof}
  
Write $\bE^{\circ}$ for the subspace of vectors in $\bX$ perpendicular to $\bE$.  Write $P_{\bE}$ and $P_{\bE^\circ}$ for the orthogonal projections from $\bX$ to $\bE$ and $\bE^{\circ}$, respectively, viewed as endomorphisms of $\bX$.
Of course, $P_{\bE}+P_{\bE^{\circ}}$ is the identity on $\bX$.

\begin{defn} Let $\bH^{\bE}$ (resp., $\bH^{\bE^\circ}$) be the set of elements in $\bH$ which fix $\bE$ (resp., $\bE^\circ$) pointwise.
  Write $\bM_E$ for the (direct) product $\bH^{\bE^\circ} \bH^E$ in $\bH$.
\end{defn}

Note that $\bS$ is a maximal torus in $\bH^{\bE}$.

\begin{lemma} \label{basic xi prop} 
\begin{enumerate}
\item  $\bM_{\bE}$ is the stabilizer of $\bE$ in $\bH$. 
 \item $N_{\bH}(\bS)=\bH^{\bE^\circ} N_{\bH^{\bE}}(\bS)$ and  $Z_{\bH}(\bS)=\bH^{\bE^\circ} Z_{\bH^{\bE}}(\bS)$.
\end{enumerate}
\end{lemma}
   
Of course, the connected component of $Z_{\bH^{\bE}}(\bS)$ is $\bS$.

\begin{defn} For a maximal $k$-torus $\bS$, let 
\begin{equation*}
\bS_{\reg}= \{ \gm \in \bS \mid \alpha(\gm) \neq 1 \text{ for all } \alpha \in R(\bH^{\bE},\bS) \}.
\end{equation*}
  Put 
\begin{equation*}
\bS_{\pm 1}=\{ \gm \in \bS \mid \text{the restriction }\gm \pm 1|_{\bE^\circ} \text{ is invertible} \},
\end{equation*}
and finally $\bS_r=\bS_{\reg} \cap \bS_{-1}$.
\end{defn}

Then $\bS_r$ is a nonempty open subset of $\bS$.  Fix a Borel subgroup $\bB=\bT \bU$ of $\bH$, with $\bT$ a maximal torus in $\bH$, and $\bU$ the unipotent radical of $\bB$. 
 Put $\bB_k=\bB \cap \bH_k$ and $\bT_k=\bT \cap \bH_k$.

\begin{defn}
Write $\bH^{\bS}$ for the union of the $\bH$-conjugates of $\bS_r$ in $\bH$, and put $\bT^{\bS}=\bH^{\bS} \cap \bT$ and $\bB^{\bS}=\bH^{\bS} \cap \bB$.
\end{defn}

To prove that $\bH^{\bS}$ is open in $\bH_k$, we follow the classic proof of the density of regular semisimple elements, due to Steinberg.  

\begin{lemma}
\begin{enumerate}
\item $\bT^{\bS}$ is open in $\bT_k$.
\item $\bB^{\bS}$ is open in $\bB_k$.
\end{enumerate}
\end{lemma}
\begin{proof}
The first part follows from Lemma \ref{bounce}, since $\bS_{r}$ is open in $\bS$. 
For the second part, consider the projection map $p_{\bT}: \bB_k \to \bT$.  We claim that $\bB^{\bS}$ is the inverse image of $\bT^{\bS}$ under $p_{\bT}$.  Let $b=tu \in \bB^S$.  Since $b$ is semisimple, it is conjugate (in $\bB$) to $t$, and thus $t \in \bT^{\bS}$.  On the other hand, suppose $b=tu \in \bB_k$ with $t \in \bT^{\bS}$.  Recall that a linear transformation is semisimple if and only if the geometric multiplicity of every eigenvalue is equal to its algebraic multiplicity.
Then 
\begin{equation*}
\rank(b-1)=\rank(t-1)=k
\end{equation*}
and $\det(b-1)=\det(t-1)$.  Moreover since $t$ is conjugate to an element of $\bS_r$, every other eigenvalue of $t$ and $b$ has multiplicity one.   We conclude that $b$ is semisimple.  Thus $b$ is conjugate to $t$ and therefore $b \in \bB^{\bS}$.  From the claim and the first part it follows that the subset $\bB^{\bS}$ is open in $\bB_k$.
 \end{proof}

 \begin{prop} \label{Zariski} $\bH^{\bS}$ is nonempty and open in $\bH_k$.
 \end{prop}
 
\begin{proof}
$\bH^{\bS}$ is nonempty because it contains $\bS_r$.

To prove it is open, let $\bA=\bB_k - \bB^{\bS}$.  Then $\bA$ is a closed subset of $\bB_k$ which is normalized by $\bB$.  Let 
\begin{equation*}
\mathcal X= \{ (h\bB,x) \in \bH/\bB \times \bH_k \mid h^{-1}xh \in \bA \},
\end{equation*}
and write $\mathcal Y$ for the image of the projection map $p_2: \mathcal X \to \bH_k$.
Then $\mathcal X$ is a closed subset of $\bH/\bB \times \bH_k$.
The map $p_2$ is proper since $\bH/\bB$ is complete, and therefore $\mathcal Y$ is a closed subset of $\bH_k$.
It is clear that $\mathcal Y = \bH_k-\bH^{\bS}$, and the result follows.
\end{proof}

\section{A Section of $\Norm$}    \label{Section_Section}
 We now return to the situation of Section \ref{Norm_Corr}.
 Let $k=\dim W$ and let $S$ be a maximal $k$-torus in $H$.  We have defined a nondegenerate subspace $E$ of codimension $k$ in the previous section, so that $S$ is a maximal torus in
 $H^E$.
    
 \subsection{Relating $X$ to $W$}
Fix a map $\xi:X \to W$ such that ${\rm ker}\ \xi=E$.  It is necessarily a surjection.  We collect some properties of $\xi$.  

\begin{lemma}\label{kerxit} \label{invert}
\begin{enumerate}
\item $\xi^*: W' \to X$ is injective with image $E^\circ$.
\item $\xi \xi^*: W' \to W$ is an isomorphism.
\end{enumerate}
\end{lemma}
\begin{proof}
 The first statement is immediate.  For the second, it is enough to prove $\xi \xi^*$ is injective.
  Suppose $\xi\xi^{*}(w')=0$ for some $w'\in W'$. For any $x\in E^{\circ}$, choose a $x'\in W'$ such that $x=\xi^{*}(x')$. Thus,
$$\Phi(\xi^{*}(w'), x)= \Phi(\xi^{*}(w'), \xi^{*}(x'))=\Phi(\xi\xi^{*}(w'), x')=0.$$Since $E^{\circ}$ is nondegenerate, $\xi^{*}(w')=0$.  Since $\xi^{*}$ is injective we have $w'=0$. 
\end{proof}

\begin{defn} Define $\ups=(\xi\xi^{*})^{-1}: W \to W'$ and $\xi^+= \xi^* \ups=\xi^{*}(\xi\xi^{*})^{-1}$.
\end{defn}

 Thus $\xi^{+}$ is a right inverse of $\xi$. 
 It is easy to see that $\xi^+ \xi=P_{E^{\circ}}$.   
 
\begin{defn}\label{xidef}
Define a map $\Xi: \End(X) \to \End(W)$ by $\Xi(A)=\xi A\xi^+$. Also define $\Xi^{+}: \End(W) \to \End(X)$ by $\Xi^{+}(A)=\xi^{+} A\xi$. 
\end{defn}

We will occasionally write ${}^\xi A=\Xi(A)$ for brevity.

\begin{prop}\label{propstab}
\begin{enumerate}
\item Let $A,B \in \End(X)$.  If $A$ or $B$ commute with $P_E$, then $\Xi(AB)=\Xi(A)\Xi(B)$.
\item $\Xi|_{M_E}$ is a group homomorphism from $M_E$ to $G$ with kernel $H^{E^\circ}$. 
\item $\Xi$ restricts to an injection from $H^E$ into $G$.
\end{enumerate}
\end{prop}

(Also see Figure \ref{Fig_1} in Section \ref{Picture_Pages}.)

\begin{defn}
 Write $\tau: \End(W) \to \End(W)$ for the antiinvolution
$$\tau(A)=\upsilon^{-1}A^{*}\upsilon,$$
and $\theta: G \to G$ for the involution $\theta(g)=\tau(g)^{-1}$.

Define a bilinear form $\Psi_{W}$ on $W$ via
\begin{equation*}
\begin{split}
\Psi_{W}(w_1,w_{2}) &=\Phi(\xi^+ w_1,\xi^+w_2) \\	
					&= \Phi(w_1, \upsilon w_2),\\
\end{split}
\end{equation*}
for $w_1,w_2 \in W$.
\end{defn}
 
\begin{prop}
\begin{enumerate}
\item For all $A \in \End(X)$, we have $\tau(\Xi(A))=\Xi(A^*)$.
\item The form $\Psi_W$ is nondegenerate and $\Psi_{W}(gw_{1},w_{2})=\Psi_{W}(w_{1},\tau(g)w_{2})$.
\item The group $G^\theta$  of fixed points is equal to $\Isom(W,\Psi_W)$.
\end{enumerate}
\end{prop}

Let ${}^\xi S=\Xi(S)$, a maximal torus in $G^{\theta}$. Set $S_{G}=Z_{G}({}^\xi S)$.  Note that
\begin{equation*}
 {}^\xi \bS=(\bS_{\bG} \cap \bG^{\theta})^\circ.
\end{equation*}

 \subsection{A section of $\Norm$ over $H^E$} 

\begin{lemma} Let $A \in \End(X)$ so that $A$ commutes with $P_E$ and $A|_{E^\circ}$ is invertible.  Then $\Xi(A)$ is invertible.
\end{lemma}

\begin{proof} Define $B \in \End(X)$ so that $B|_{E^\circ}=(A|_{E^\circ})^{-1}$ and $B|_E=0$.  Since $B$ commutes with $\xi^+ \xi= P_{E^\circ}$, we see that $\Xi(B)$ is the inverse of $\Xi(A)$.
\end{proof}

\begin{defn}  
Let $h \in M_E$ with $h-1|_{E^\circ}$ invertible.  
Define $h_G=\left( \Xi(h-1) \right)^{-1} \in G$.
\end{defn}

\begin{prop} \label{About_h_G}
Let $h \in M_E$ with $h-1|_{E^\circ}$ invertible.  We have
\begin{enumerate}
\item $1+h_G+\tau(h_G)=0$.
\item $h_G \cdot \theta(h_G)=-\Xi(h^{-1})$.
\item If $h$ centralizes $S$, then $h_G \in S_G$.
\end{enumerate}
\end{prop}

{\bf Remark:}  This second property points to the usage of the word ``Norm'' in this subject.
 
\begin{proof}
For the first statement, multiply the expression
\begin{equation*}
1+(\Xi(h-1))^{-1}+(\Xi(h^{-1}-1))^{-1}
\end{equation*}
by $\Xi(h-1)\Xi(h^{-1}-1)$ to obtain
\begin{equation*}
\Xi(h-1)\Xi(h^{-1}-1)+\Xi(h^{-1}-1)+\Xi(h-1)=0.
\end{equation*}
The second statement is similar, and the third is easy.
\end{proof}

We can now construct the section that we want.  

\begin{defn}  
Let $h \in H^E$ with $h-1|_{E^\circ}$ invertible.  Let  $\eta_S(h)=h_G \ups^{-1}$.
\end{defn}
Thus $\eta_S(h): W' \to W$ is given by
\begin{equation*}
\eta(h)=\eta_S(h)=(\xi(h-1)\xi^+)^{-1} \xi \xi^*.
\end{equation*}

{\bf Remark:}  One could also write, for example, $\eta(h)=\xi (h-1-P_E)^{-1}\xi^*$.

\bigskip

Finally we define
\begin{equation*}
n(h)=n_{S}(h)=n(\xi,\eta_S(h)).
\end{equation*}
 
Note that any $h=\gm \in S_{r}$ satisfies the hypothesis.

\begin{prop} \label{section} Let $h \in H^E$ with $h-1|_{E^\circ}$ invertible.  Then $n_S(h) \in N'$ and $\Norm(n_S(h))=h$.
\end{prop}

\begin{proof}
The pair $(\xi,\eta(h))$ satisfies (\ref{defn}) by Proposition \ref{About_h_G}.  We compute
\begin{equation*}
\begin{split}
\Norm(n(h)) &= 1+\xi^* \eta(h)^{-1} \xi \\
			&=1+ P_{E^\circ} h P_{E^\circ}-P_{E^\circ} \\
			&= h, \\
\end{split}
\end{equation*}
as desired.  
\end{proof}
 
\subsection{Matrices} \label{MatrixSection}
In this section we describe the map $n \mapsto n(\gm)$ in terms of matrices.  For simplicity, we assume that $\Phi$ is split symmetric or antisymmetric.
Suppose $D=\dim V$, $2n=\dim W$, and $r=\dim E$.  Thus $\dim X=r+2n$ and $D=r+6n$.

For a number $m$, write $J_+(m)$ for the $m \times m$ symmetric matrix
\begin{equation*}
J_+(m)= \left(\begin{array}{ccccc}
&&&&1\\
&&&1&  \\
&&...&&\\
&1&&&\\
1&&&&
\end{array}\right).
\end{equation*}

For an even number $m$, write $J_-(m)$ for the $m \times m$ antisymmetric matrix
\begin{equation*}
J_-(m)= \left(\begin{array}{ccccc}
&&&&1\\
&&&-1&  \\
&&...&&\\
&1&&&\\
-1&&&&
\end{array}\right).
\end{equation*}
 
In the symmetric case (resp. the antisymmetric case) choose ordered bases $w_{1},...,w_{2n}$ and $w_{1}',...,w_{2n}'$ of $W$ and $W'$ such that $\Phi(w_{i},w_{j}')=\delta_{ij}$ (resp., $\Phi(w_{i},w_{j}')=(-1)^j\delta_{ij}$).
 
 Pick a basis $e_1, \ldots, e_r$ of $E$ so that $J_+(r)$ (resp. $J_-(r)$) is the matrix of $\Phi_E$ with respect to this basis, and a basis $f_1, \ldots, f_n, f_1', \ldots, f_n'$ of $E^\circ$ so that
$\Phi(f_{i},f_{j}')=\delta_{ij}$ (resp., $\Phi(f_{i},f_{j}')=(-1)^j\delta_{ij}$). We will use the ordered basis 
\begin{equation*}
w_{1},...,w_{2n}, f_1, \ldots, f_n, e_1, \ldots, e_r, f_n', \ldots, f_1', w_{2n}', \ldots, w_1'
\end{equation*}
of $V$, and the ordered basis $ f_1, \ldots, f_n, e_1, \ldots, e_r, f_n', \ldots, f_1'$ of $X$.
 Then $J_{\pm }(D)$ is the matrix of $\Phi$ with respect to this basis.
 
Define $\xi: X \to W$ by $\xi(f_i)=w_i$, $\xi(f_i')=w_{2n+1-i}$ and $\xi|_E=0$.  We represent $\xi$ with the matrix $\xi=\left(\begin{array}{ccc}
I_{n}&0&0\\
0&0&I_{n}  
\end{array}\right)$, and $\xi^*: W' \to X$ with the matrix  \newline $\xi^{*}=\left(\begin{array}{cc}
I_{n}&0\\
0&0\\
0&I_{n}  
\end{array}\right)$.  Next, consider the torus $S$ consisting of matrices of the form
Let 
\begin{equation*}
\gamma=\left(\begin{array}{ccccccc}
t_{1}&&&&&&\\
&\ddots&&&&&  \\
&&t_{n}&&&&\\ 
&&&I&&&\\
&&&&t_{n}^{-1}&&\\
&&&&&\ddots&\\
&&&&&&t_{1}^{-1}
\end{array}\right),
\end{equation*}
 relative to the above basis of $X$.
We then have
\begin{equation*}
\gamma_G=\left(\begin{array}{cccccc}
(t_{1}-1)^{-1}&&&&&\\
&\ddots&&&&  \\
&&(t_{n}-1)^{-1}&&&\\ 
&&&(t_{n}^{-1}-1)^{-1}&&\\
&&&&\ddots&\\
&&&&&(t_{1}^{-1}-1)^{-1}
\end{array}\right),
\end{equation*}
relative to the above basis of $W$.

Finally, our matrix $n_S(\gamma) \in N$ is written by fitting together the above matrices via

\begin{equation*}
n_S(\gm)=\left(\begin{array}{ccc}
I& \xi &\gm_G \\
&I&-\xi^*  \\
&&I
\end{array}\right).
\end{equation*}
(Our choice of basis obviates the need for the factor $\upsilon^{-1}$.)
 
\section{The $\Sha_S$ map: Algebraic Theory} \label{alg_theory}

\subsection{Decomposition of $N'$}

Let $S$ be a maximal $k$-torus of $H$ defined over $F$.  

\begin{defn} Put
\begin{equation*}
N^S= \{ n(\xi,\eta) \in N' \mid \Norm(\xi,\eta) \text{ is conjugate to an element of }  S_r \}.
\end{equation*}
\end{defn}
 
\begin{defn} Define the map $\Sha_S: M \times S_r \to N$ by
\begin{equation} \label{fake_Sha}
\begin{split}
\Sha_S(m(g,h) \times \gamma) &= \Ad(m(g,h))n_S(\gm)\\
							&= n(g\xi h^{-1},g\eta_S(\gamma)g^{*})\\
							&=n(g \xi h^{-1}, g (\gm)_G \ups^{-1} g^*). \\
\end{split}
\end{equation}
\end{defn}

We will freely drop the subscript ``$S$'' from  $\eta_S$, $n_S$, and $\Sha_S$ when it becomes cumbersome.

\begin{prop}\label{surjprop}  We have
\begin{enumerate}
\item $\Norm(\Sha_S(m(g,h) \times  \gm))=  h \gm h^{-1}$.
\item The image of $\Sha_{S}$ is equal to $N^S$.
\end{enumerate}
\end{prop}

\begin{proof}
The first statement follows from Lemma \ref{basicnorm} and Proposition \ref{section}.  It follows that the image of $\Sha_S$ is contained in $N^S$.
 
Now let $n=n(\xi_{1},\eta_{1})\in N^{S}$.  Then there are $h \in H$ and $\gm \in S_r$ so that $h(\textrm{Norm}(\xi_{1},\eta_{1}))h^{-1}=\gamma$.  Since we also have $\gm=\Norm(\xi,\eta(\gm))$, we obtain the equation
\begin{equation*}
h\xi_{1}^{*}\eta_{1}^{-1}\xi_{1}h^{-1}=\xi^{*}\eta(\gamma)^{-1}\xi.
\end{equation*}
Thus
\begin{equation*}
\xi h\xi_{1}^{*}\eta_{1}^{-1}\xi_{1}h^{-1} \xi^*=\xi \xi^{*}\eta(\gamma)^{-1}\xi \xi^*.
\end{equation*}
Now the right hand side is an isomorphism, and therefore the map $\xi_{1}h^{-1} \xi^*: W' \to W$ is an isomorphism.  Put $g=\xi_{1}h^{-1} \xi^+$; we obtain
\begin{equation*}
\eta_1=g \eta(\gm) g^*.
\end{equation*}

Since $g\xi$ and $\xi_{1}$ are both surjective maps from $X$ to $W$, there exists $g_{1}\in \GL(X)$ such that $g\xi=\xi_{1}g_{1}$.
Since $n(g\xi,\eta_{1})\in N$ (by Lemma \ref{basicnorm}), we have
\begin{equation*}
\begin{split}
\xi_{1}\xi_{1}^{*} &=-(\eta_{1}+\eta_{1}^{*}) \\
			&= (g\xi)(g\xi)^{*} \\
			&= (\xi_{1}g_{1})(\xi_{1}g_{1})^{*}. \\
\end{split}
\end{equation*}
			 By Lemma \ref{gentoclass}, there is an $h_1\in H$ such that $\xi_{1}g_{1}=\xi_{1}h_1$. Thus, $\xi_{1}=g\xi h_1^{-1}$.
It follows that $\Sha_S(m(g,h_1) \times \gm)=n(\xi_1,\eta_1)$, and therefore $N^S$ is contained in the image of $\Sha_S$.
\end{proof}

Let us pause to appreciate some progress we have made in our decomposition of $N$.

\begin{defn} Write $\bN_r$ for the preimage of the $\Norm$ map in $\bN$ of $\bH^{\bS}$, and $N_r$ for the $F$-points of $\bN_r$.
\end{defn} 

\begin{thm} \label{Theorem 1}
\begin{enumerate}
\item $\bN_r$ is nonempty and open in $\bN$.
\item $N_{r}$ is the union of $N^S$, as $S$ runs over the maximal $k$-tori $S$ of $H$.
\item We have a decomposition
\begin{equation} \label{N_decomp}
N_{r}= \bigcup_S \{ \Ad(m) n_S(\gm) \mid m \in M, \gm \in S_r \}.
\end{equation}
\end{enumerate}
\end{thm}

\begin{proof} 
Recall that the image of $\Norm$ lies in $\bH_k$.
Since $\bH^{\bS}$ is open in $\bH_k$, we see that $\bN_r$ is open in $\bN$.  One can apply $n_{\bS}$ to elements of $ \bS$ to produce elements of $\bN_r$, so it is nonempty.
The second statement is clear, and the last statement follows from Proposition \ref{surjprop}.
\end{proof}

\subsection{Fibre of $\Sha_S$}

There is some redundancy in the decomposition (\ref{N_decomp}), coming from  the normalizer $N_H(S)$ of $S$ in $H$.  

\begin{defn} Define $\Delta: M_E \to M$ via $\Delta( h)=m(\Xi( h), h)$.  Write $\Delta_S$ for the image $\Delta(Z_H(S))$ of $Z_H(S)$ under $\Delta$.
\end{defn}

Consider the action of $N_H(S)$ on $M \times S_{r}$ given by:
\begin{equation*}
w : (m \times \gm) \mapsto m \Delta(w) \times w^{-1} \gm w,
\end{equation*}
for $w \in N_H(S)$.  This descends to an action of $W_H(S)$ on $M/\Delta_S \times S_{r}$.

  \begin{prop} Let $m \in M$ and $\gm \in S_{r}$.  The fibres of $\Sha_S$ are the same as the $N_H(S)$-orbits on $M \times S_{r}$.  In other words, the fibre of $\Sha_S$ containing $m \times \gm$ is precisely
\begin{equation*}
\{ m \Delta( w) \times  w^{-1} \gm  w \mid  w \in N_H(S) \}.
\end{equation*}
\end{prop}
  
\begin{proof}  
The reader may check that these elements are indeed in the fibre.  Suppose now that there are $g,g' \in G$, $h,h' \in H$, and $\gm, \gm' \in S_{r}$ so that
\begin{equation} \label{labeled}
\Sha_S(m(g,h) \times \gm)=\Sha_S(m(g',h') \times \gm').
\end{equation}
 Thus $n(g\xi h^{-1},g\eta(\gamma)g^{*})=n(g'\xi (h')^{-1},g'\eta(\gamma')(g')^{*})$. Comparing the first coordinates we find $(g')^{-1}g \xi =\xi (h')^{-1}h$. Let $x=(g')^{-1}g$ and $w=(h')^{-1}h$. Then we find $x \xi=\xi w$, and therefore 
 $x=\Xi(w)$.

Taking $\Norm$ of both sides of (\ref{labeled}) gives
\begin{equation*}
h \gm h^{-1}=h' \gm' (h')^{-1}
\end{equation*}
by Proposition \ref{surjprop}. Since $\gm, \gm' \in S_{\reg}$ it follows that $w \in N_H(S)$.  The result follows.
 \end{proof} 
 
\begin{cor} The map $\Sha_{S}$ descends to a surjective map  
\begin{equation} \label{real_Sha}
\Sha_{S}: M/\Delta_S \times S_{r}\to N^{S}.
\end{equation}
The fibres of $\Sha_S$ are the same as the $W_H(S)$-orbits on $M/\Delta_S \times S_{r}$.  In other words, the fibre of $\Sha_S$ containing $m \times \gm$ is precisely
\begin{equation*}
\{ m \Delta( w) \times  w^{-1} \gm  w \mid  w \in W_H(S) \}.
\end{equation*}
\end{cor}

The map defined by (\ref{real_Sha}) is the one we will study henceforth.

\section{Lie Algebras} \label{Lie}

As usual, we use the Fraktur analogues of the Latin font to denote the Lie algebras of a given group. Thus $\GG$, $\hh$, $\mm$, $\nn$, $\LS$, etc. will denote the Lie algebras of $G$, $H$, $M$, $N$, $S$, etc. respectively.  In particular, 
$z_{\hh}(\LS)$ will denote the Lie algebra of $Z_{H}(S)$.

An element $u\in \mathfrak n$ is characterized by linear maps
$$A:X\to W, \ A':W'\to X, \ B:W'\to W$$
such that:
\begin{enumerate} 
\item $u(w)=0 \ \ \forall w\in W$.
\item If $x\in X$ then $u(x)=A(x)$. 
\item If $w'\in W'$ then $u(w')=A'(w')+B(w')$.
\end{enumerate}
The condition that $u\in \mathfrak n$ is equivalent to the two conditions 
\begin{enumerate}
\item $A^{*}+A'=0$,
\item $B^{*}+B=0$. 
\end{enumerate}
So as usual we write $u\in \mathfrak n$ as $u(A,B)$ where the condition is simply that $B$ is skew-Hermitian on $W'$. 
 
By Lemma \ref{basic xi prop}, we have 
\begin{equation} \label{blabel}
z_{\hh}(\LS)=\hh^{E^\circ} + \LS,
\end{equation}
and the sum is direct.

\subsection{Four Exact Sequences} \label{4X}

The relationships between these Lie algebras are subtle, governed by no fewer than four exact sequences.  The starting point is the map
 $\phi:\mm \to \mathfrak n$ given by
 \begin{equation*}
 \phi(A,B)=u(A\xi,\xi B\xi^{*}).
 \end{equation*}
This map is typically neither injective nor surjective.  We identify its kernel and image.
 
\begin{defn}
Let 
\begin{equation*}
\kappa= \{ B \in \hh \mid B(E^\circ) \subseteq E \}.
\end{equation*}
Define $i_{2}:\kappa\to \mathfrak m$ by $i_{2}(B)=(0,B)$.
Also we set
$$\mathfrak n_{\xi}=\{u(A,B) \in \nn \mid  E\subseteq \ker A\}.$$
\end{defn}

\begin{prop}  (The Exact Sequence for $\mm$) The following sequence is exact:
\begin{equation} \label{M_seq}
0 \to \kappa\overset{i_{2}}\to \mathfrak m \overset{\phi}\to \mathfrak n_{\xi} \to 0.
\end{equation}
\end{prop}
  
We next identify the quotient of $\nn$ by $\nn_{\xi}$.

\begin{defn}
 Let $i$ denote the inclusion of $\mathfrak n_{\xi}$ into $\mathfrak n$. 
Write $\psi: \mathfrak n \to \Hom(W',E)$ for the map given by
\begin{equation*}
\psi: u(A,B)\mapsto P_{E}\circ A^{*}.
\end{equation*}
\end{defn}

\begin{prop} \label{N_exact} (The Exact Sequence for $\nn$) The following is a split exact sequence:
\begin{equation} \label{N_seq}
0\to \mathfrak n_{\xi}\overset{i}\to \mathfrak n \overset{\psi}\to \Hom(W',E) \to 0.
\end{equation}
\end{prop}
Define $\tilde \psi: \Hom(W',E) \to \nn$ via $\tilde \psi(A) =u(A^*P_E,0)$; then $\tilde \psi$ is a section of $\psi$.
  \bigskip
  
The space $\kappa$ also maps onto $\Hom(W',E)$.  
  
 \begin{defn}  Define $\Gamma_{1}:\kappa\to \Hom(W',E)$ via $\Gamma_{1}(B)=B\xi^{*}$. Further define $\Gamma_{2}:\Hom(W',E)\to \hh$ via $\Gamma_{2}(A)=A(\xi\xi^{*})^{-1}\xi- \xi^{*}(\xi\xi^{*})^{-1}A^{*}$. 
\end{defn}

One checks that the image of $\Gamma_2$ lands in $\kappa$, and that $\Gamma_1 \circ \Gamma_2$ is the identity on $\Hom(W',E)$.

\begin{prop} \label{kappaX} (The Exact Sequence for $\kappa$) The following sequence is exact:
\begin{equation} \label{kappa_seq}
0 \to \hh^{E^\circ} \overset{i}\to \kappa \overset{\Gamma_{1}}\to \Hom(W',E) \to 0.
\end{equation}
Here $i$ is the natural inclusion.
\end{prop}
 
Of course, by (\ref{blabel}), $\hh^{E^\circ}$ is the quotient of $z_{\hh}(\LS)$ by $\LS$.  We make this explicit as follows:

\begin{defn}  Define the map $K:z_{\mathfrak h}(\mathfrak s) \to \hh^{E^\circ}$ such that $K(A)=AP_{E}$.
\end{defn}

\begin{prop} (The Exact Sequence for $z_{\mathfrak h}(\mathfrak s)$) The following sequence is exact:
\begin{equation} \label{Z_seq}
0\to \mathfrak s\overset{i}\to z_{\mathfrak h}(\mathfrak s) \overset{K}\to \hh^{E^\circ} \to 0.
\end{equation}
Here $i$ is the natural inclusion.
\end{prop}
 
We will later use these exact sequences to calibrate our differential forms on these spaces.  For the moment, we use them to make a simple observation about dimensions.

\begin{cor} \label{dimension_formula}
\begin{equation*}
\dim M-\dim Z_H(S)+ \dim S=\dim N.
\end{equation*}
 \end{cor}
 
 \subsection{Decompositions} \label{Picture_Pages}
 
\begin{prop} We have a direct sum decomposition
\begin{equation*}
\hh=\hh^E + \hh^{E^\circ}+ \im \Gamma_2.
\end{equation*}
\end{prop}

\begin{proof}
If $X \in \hh$, then $P_EXP_E+P_{E^\circ}XP_{E^\circ} \in \mm_E= \hh^E + \hh^{E^\circ}$, and the difference $B=X- (P_EXP_E+P_{E^\circ}XP_{E^\circ})$ is easily seen to lie in $\kappa$.  
Moreover a computation shows that $B=\Gamma_2(\Gamma_1(B))$. 
To see the sum is direct, first note that $\kappa \cap \mm_E=\hh^{E^\circ}$.  Next, suppose that $\Gamma_2(A) \in \hh^{E^\circ}$.  Then $\Gamma_1(\Gamma_2(A))=A=0$.  It follows that the intersection of $\im \Gamma_2$ with $\mm_E$ is trivial.
\end{proof}

Much of this can be visualized with matrix geometry; we offer the following picture:

\colorlet{lightgray}{black!15}
\begin{center}
\begin{tikzpicture}
 \fill[color=lightgray] (-1,3) rectangle (1,-3);
 \fill[color=lightgray] (-3,1) rectangle (3,-1);
 \fill[color=gray] (-1,1) rectangle (1,-1);
 
 \draw (-1,3) -- (-1,-3);
 \draw (1,3) -- (1,-3);
 \draw (-3,-1) -- (3,-1);
 \draw (-3,1) -- (3,1);
 \draw[line width=0.05cm] (-3,3) arc (165:195:12);
 \draw[line width=0.05cm] (3,-3) arc (-15:15:12);
 %\draw[style=dashed] (-3,-3) -- (3,3);
 \draw[style=dashed] (-3,3) -- (3,-3);
 \path (-2,3.5)node (A) {$E^{\circ}$};
 \path (0,3.5)node (B) {$E$};
 \path (2,3.5)node (C) {$E^{\circ}$};
 \path (-4,-2)node (D) {$E^{\circ}$};
 \path (-4,0)node (E) {$E$};
 \path (-4,2)node (F) {$E^{\circ}$};
 \draw (A); \draw (B); \draw (C); \draw (D); \draw (E); \draw (F);
\end{tikzpicture}
\captionof{figure}{Composition of $\hh$}
 \label{Fig_1}

\end{center}

Figure \ref{Fig_1} represents $\hh$, viewed as matrices in $\GG l(X)$.  The middle rows and columns corresponding to $E$ are shaded (both light gray and dark gray); this is $\kappa$.  The subalgebra $\hh^{E^\circ}$ corresponds to the dark gray central block. The Lie algebra $z_{\HH}(\LS)$ is the sum of the dark gray block and the diagonal. The light gray region corresponds to the image of $\Hom(W',E)$ under $\Gamma_{2}$.  
The unshaded region is $\hh^E$.
 The map $\Xi$ can be visualized by simply deleting the five shaded regions.  
 
 \bigskip
 
Meanwhile, $\nn$ can be expressed as a direct sum $\nn=\nn_\xi + \im \tilde \psi$, by Proposition  \ref{N_exact}.
 The following is a diagram of an essential piece of $\nn$:

\colorlet{lightgray}{black!15}
\begin{center}
\begin{tikzpicture}  
 \fill[color=lightgray] (-1,3) rectangle (1,1);
  
 \draw (-1,3) -- (-1,1);
 \draw (1,3) -- (1,1);
 
 \draw (3,3) --(3,1);
  
 \draw[line width=0.05cm] (-3,3) arc (165:195:4);
 
    \draw[line width=0.05cm] (5,1) arc (-15:15:4);

 \path (-2,3.5)node (A) {$E^{\circ}$};
 \path (0,3.5)node (B) {$E$};
 \path (2,3.5)node (C) {$E^{\circ}$};
 \path (4,3.5) node (D) {$W'$};
 \path (4,2)node (D) {$B$};
 %\path (-4,0)node (E) {$W$};
 \path (-4,2)node (F) {$W$};
 \draw (A); \draw (B); \draw (C); \draw (D); \draw (E); \draw (F);
\end{tikzpicture}

\captionof{figure}{Blow-up of $\nn$}
 \label{Fig_2}
 \end{center}
 
Figure \ref{Fig_2} is a blow-up of the $(A,B)$ portion of
$\left(\begin{array}{ccc}
0& A &B\\
&0&  -A^*\\
&&0
\end{array}\right) \in \nn.
$
The shaded region corresponds to $\Hom(W',E) \cong \Hom(E,W)$, or the image of $\tilde \psi$.  The unshaded area corresponds to $\nn_\xi$.
 The skew-symmetric matrices $B$  comprise the image of $ 0 \oplus \hh_{E^\circ}$ under $\phi$.  The rest of the unshaded area corresponds to the image of 
 $\GG \oplus 0 \subset \mm$ under $\phi$.

\section{The $\Sha$ map:  Differential Theory} \label{local_now}

\subsection{Derivative of $\Sha$}
Now suppose that $F$ is a local field.  In this paper ``manifold'' means a smooth finite-dimensional $F$-manifold in the sense of \cite{BourbDiff}.
If $X$ is a manifold and $p\in X$, we write $T_pX$ for the tangent space to $X$ at $p$.

Recall our map
\begin{equation*}
\Sha=\Sha_S: M/\Delta_S \times S_r \to N.
\end{equation*}

The derivative
\begin{equation*}
d\Sha_{1 \times \gamma} : T_1( M/\Delta_S) \oplus T_\gm S \to T_{n(\gm)}N.
\end{equation*}
 is straightforward to compute.

\begin{prop} For $A\in \mathfrak g$, $B\in \mathfrak h$ and $Z\in \mathfrak s$, 
we have
\begin{equation} \label{the_expression}
d\Sha_{1 \times \gamma}(A \times B,\gamma Z)=u(A\xi-\xi B, \eta(\gamma)A^{*}+A\eta(\gamma)- \gm_G^2 \Xi(\gm Z) \ups^{-1}).
\end{equation}
\end{prop}

For later use we note:
\begin{lemma}\label{zetaim}
$\gm_G^2 \Xi(\gm Z) \in {}^\xi \LS$.
\end{lemma}

The derivative at $m \times \gamma \in  M/\Delta_S \times S_r$ can be inferred from (\ref{the_expression}) through the diagram

 \begin{equation} \label{trans_diagram}
	\xymatrixcolsep{5pc}\xymatrix{
		  T_1( M/\Delta_S) \oplus T_\gm S   \ar[r]^-{ d\Sha_{1 \times \gm}} \ar[d]_{\rho_m} &T_{\Sha(1 \times \gm)} N\ar[d]^{\Ad(m)} \\
		T_m( M/\Delta_S) \oplus T_\gm S  \ar[r]^-{ d\Sha_{m \times \gm}}& T_{\Sha(m \times \gm)} N \\
		}
\end{equation}

Here $\rho_m$ denotes right translation by $m \in M$.

For later use, we multiply (\ref{the_expression}) on the right by $n(\gamma)^{-1}$ to bring it to $T_1N=\nn$.  
This gives
$$\textrm{d}\Sha_{1\times1,\gamma}(A\times B, \gamma Z)n(\gamma)^{-1}=$$
\begin{align}\label{shadiff}
u(A\xi-\xi B, (A\xi-\xi B)\xi^{*}+\eta(\gamma)A^{*}+A\eta(\gamma)- \gm_G^2 \Xi(\gm Z) \ups^{-1}).
\end{align}

The next step is to compute a suitable Jacobian of $\Sha$.

\subsection{Differential Forms and Measures}
 
Let us recall the definition of differential forms and their associated measures.

\begin{defn} If $V$ is a vector space over $F$, write $\Alt^d(V)$ for the space of alternating forms on $d$-tuples of vectors in $V$.
\end{defn}
One denotes by $\underline{v}$ a $d$-tuple $(v_{1},...,v_{d}) \in V^d$.  
We will later employ the following construction from Chapter IX, Section 6 of \cite{BourbLie}:
\begin{defn}
Let
$$0\to U'\overset{i}\to U \overset{p}\to U'' \to 0$$
be an exact sequence of vector spaces, of dimension $s$,$s+t$, and $t$, respectively. Let $\alpha''\in \Alt^{t}(U'')$ and $\alpha'\in\Alt^{s}(U')$. Then there is an alternating form $\alpha''\cap \alpha'\in\Alt^{s+t}(U)$, which is characterized by the following property: If $\underline{v}\in U^{t}$ and $\underline{v'}\in (U')^{s}$, then 
$$\alpha''\cap \alpha'(\underline{v},i(\underline{v'}))= \alpha''(p(\underline{v}))\alpha'(\underline{v'}).$$   
\end{defn}

\begin{defn}
Let $X$ be a manifold.  A (differential) $n$-form $\omega$ on $X$ is a smooth choice of alternating forms $\omega(p)\in \Alt_{n}(T_{p}X)$ for each point $p\in X$.
If $n=\dim X$, then an $n$-form is called a top form.
\end{defn}

Now suppose that $\omega$ is a top form, and put $n=\dim X$.  If $u^1,\ldots, u^n$ are coordinates on an open subset $U$ of $X$, then there is a smooth function $f$ on $U$ so that $\omega |_U$ is the form $f du^1 \wedge \cdots \wedge du^n$.  Then a real-valued measure $|\omega|$ on $X$ may be assembled by combining $|f|$ with the product of fixed Haar measures on the additive group of $F$ via the $u^i$.  (For details see \cite{BourbDiff}.)  In this case $|\omega|$ is called the measure associated to $\omega$.  If $G$ is a Lie group, this gives a one-to-one correspondence between left-invariant differential forms $\omega_G$ on $G$ (up to a nonzero constant in $F$) and left Haar measures $dg=|\omega_G|$ on $G$ (up to a positive constant in $\R$).

Suppose that $Y$ is another manifold. If $h : Y \to X$ is smooth, then the pullback form $h^{*}(\omega)$ on $Y$ is an $n$-form defined by the equation
\begin{equation*}
h^{*}(\omega)(p)(v_{1},...,v_{n})=\omega(h(p))(dh_{p}(v_{1}),...,dh_{p}(v_{n})).
\end{equation*} 
 
We will later use the following result:

\begin{prop}\label{difffibre} 
Let $d\geq 1$. Suppose that $dh$ does not vanish at any point of $Y$, and that the preimage of each point of $X$ has precisely $d$ points. 
Let $f \in L^1(X,|\omega|)$.  Then $f \circ h \in L^1(Y,|h^{*}(\omega)|)$ and we have the identity
$$\int_{X}f(x)|\omega|=\frac{1}{d}\int_{Y}f(h(y))|h^{*}(\omega)|.$$
\end{prop}
 
 \subsection{Jacobian of $\Sha$:  First Step}
 Suppose that $\omega_N$ and $\omega_S$ are invariant top forms on $N$ and $S$, respectively, and that $\omega_{M/\Delta_S}$ is an $M$-invariant top form on $M/\Delta_S$.
(We will specify these later.)

\begin{defn} For $m \in M$, put
\begin{equation*}
\delta_N(m)=\det(\Ad(m); \nn).
\end{equation*}
\end{defn}

\begin{prop} \label{deltas}There is a smooth function $\delta$ on $S_r$ so that
\begin{equation*}
\Sha_S^*(\omega_N)=\delta_N(m)\delta(\gm) \omega_{M/\Delta_S} \wedge \omega_S
\end{equation*}
at the point $m \times \gm \in M/\Delta_S \times S_r$.
\end{prop}

\begin{proof}
Since $\Sha_S^*(\omega_N)$ and $\omega_{M/\Delta_S} \wedge \omega_S$ are both top forms on $M/\Delta_S \times S_r$, there is a unique (smooth) function
\begin{equation*}
\delta: M/\Delta_S \times S_r \to F^{\times}
\end{equation*}
so that
\begin{equation*}
\Sha_S^*(\omega_N)=\delta(m,\gm) \omega_{M/\Delta_S} \wedge \omega_S.
\end{equation*}
Let $m_0 \in M$. Call $\rho_{m_0}$ the map given by right multiplication by $m_{0}$. Applying $\rho_{m_0}^*$ to both sides of the equation gives
\begin{equation*}
\delta_N(m_0) \Sha_S^*(\omega_N) =\delta(mm_0,\gm)\omega_{M/\Delta_S} \wedge \omega_S.
\end{equation*}
(We have used $\Ad(m)^*\omega_N=\delta_N(m)\omega_N$.)
Therefore for all $m,m_0$, and $\gm$, we have
\begin{equation*}
\delta(mm_0,\gm)=\delta_N(m_0)\delta(m,\gm).
\end{equation*}
The result follows (put $\delta(\gm)=\delta(1,\gm)$).
\end{proof}
 
\subsection{Choice of Bases}  \label{JacobianSection}
To compute $\delta$, we essentially need to compute the determinant of the differential of $\Sha_S$.   One cannot literally do this, of course, because the differential is a map between two different tangent spaces.  Thus one should relate these two spaces in some natural way.   The exact sequences of Section \ref{4X} allow us to break up these two spaces into the same ``pieces'', and our entire determinant calculation is based on these decompositions.  Happily, the differential of $\Sha_S$ is an upper triangular ``block matrix'' with respect to these pieces.  In this section we set up two bases of $\nn$, namely $\mathfrak B_1$ and $\mathfrak B_2$.  The basis $\mathfrak B_1$ will essentially be the input of $d \Sha_S$, and $\mathfrak B_2$ will essentially be the output.  Then $\delta$ will be proportional to the determinant of a linear transformation $L$ sending $\mathfrak B_1$ to $\mathfrak B_2$.
 Consider a basis $\mathfrak B$ of $\mathfrak m$ of the form
$$\{(\underline{A_{i}},0),(0,\Gamma_2(\underline{\beta_{j}})),(\Xi(\underline{B'_{j'}}),\underline{B'_{j'}}),(0,\underline{C_{k}}),(\Xi(\underline{Z_{l}}),\underline{Z_{l}})\}$$
 
 such that 
\begin{itemize}
\item $\{A_{i}\}$ is a basis of $\GG$.
\item $\{\beta_{j}\}$ is a basis of $\Hom(W',E)$.
\item  $\{B_{j'}'\}$ is a basis of $\hh^{E^\circ}$.
\item $\{Z_{l}\}$ is a basis of $\mathfrak s$.
\item $\{C_{k}\} \cup \{Z_{l}\}$ is a basis of $\hh^E$.
\end{itemize}
  Note that $\{B_{j'}'\}\cup \{Z_{l}\}$ is a basis of $z_{\mathfrak h}(\mathfrak s)$.  Put $B_j=\Gamma_2(\beta_j)$. It can be easily checked that $B_j\xi^{*}=\beta_{j}$.

\bigskip
{\bf Remark:}  Later we will demand that this basis respects certain root space decompositions.
\bigskip

Then the following are two bases of $\mathfrak n$:
$$\mathfrak B_{1}=\{u(\underline{A_{i}}\xi,0),u(\underline{\beta_{j}},0),u(0,\xi\underline{C_{k}}\xi^{*}),u(0,\xi\underline{Z_{l}}\xi^{*})\}$$
\begin{equation*}
\mathfrak B_{2}=\left\{u(\underline{A_{i}}\xi,-),u(\underline{\beta_{j}},0),u(-\xi\underline{C_{k}},-\xi\underline{C_{k}}\xi^{*}),u\left(0,- \gm_G^2 \Xi(\gm Z) \ups^{-1}\right) \right\}
\end{equation*}
where
$$u(A_{i}\xi,-)=u\big( A_{i}\xi, A_{i}\xi\xi^{*}+ \eta(\gamma)A_{i}^{*} + A_{i}\eta(\gamma) \big).$$ 

To see that $\mathfrak B_1$ is a basis, note that $u(\underline{A_{i}}\xi,0)$, $u(0,\xi\underline{C_{k}}\xi^{*})$, and $u(0,\xi\underline{Z_{l}}\xi^{*})$ comprise a basis of $\nn_\xi$ and that
$u(\underline{\beta_{j}},0)$ gives a basis of $\im \tilde \psi$.

Finally, write $p\mathfrak B$ for
\begin{equation*}
p\mathfrak B= \left\{  (p(\underline{A_{i}},0);0);(p(0,\underline{B_j)};0);(p(0,\underline{C_{k}});0);(0;\underline{\gamma Z_{l}}) \right\},
\end{equation*}
a basis of $\mathfrak m/\Delta_{\LS} \oplus \LS$.  Here $p$ is the projection from $\mm$ to $\mathfrak m/\Delta_{\LS}$.

In the next section we will calibrate our forms $\omega_N$ and $\omega_{M/\Delta_S \times S}$ so that for all $\gm \in S_r$, we have
\begin{equation} \label{calibration}
\omega_{N}(1)[\mathfrak B_{1}]=\pm \omega_{M/\Delta_S \times S}(1 \times \gamma)[p \mathfrak B].
\end{equation}
Let us assume this now for the sake of exposition.

\begin{prop} \label{delta/det} Suppose that $L=L(\gm): \nn \to \nn$ is a linear transformation so that $L(\mathfrak B_{1})=\mathfrak B_{2}$.  Then $\delta=\pm  \det L$.
 \end{prop}

\begin{proof}
Recall that $\delta$ is the function on $S_r$ so that at $1 \times \delta \in M/\Delta_S \times S_r$, we have
\begin{equation*}
\Sha_S^* \omega_N= \delta(\gm) \omega_{M/\Delta_S} \wedge \omega_S.
\end{equation*}

  Now,
\begin{equation} \label{Guest}
\Sha_{S}^{*}\omega_{N}(1\times \gamma)[(p(\underline{A_{i}},0);0);(p(0,\underline{B_j});0);(p(0,\underline{C_{k}});0);(0;\underline{\gamma Z_{l}})]
\end{equation}
is equal to 
$$\omega_{N}(n(\gamma))[d\Sha(p(\underline{A_{i}},0);0);d\Sha(p(0,\underline{B_{j}});0);d\Sha(p(0,\underline{C_{k}});0);d\Sha(0;\underline{\gamma Z_{l}})].$$

By (\ref{shadiff}), we have
\begin{equation*}
\begin{split}
d\Sha(p(\underline{A_{i}},0);0) n(\gm)^{-1}&=u(\ul{A_i} \xi,\ul{A_i} \xi \xi^*+\eta(\gm) \ul{A_i}^*+\ul{A_i} \eta(\gm)), \\
d\Sha(p(0,\underline{B_{j}});0)n(\gm)^{-1}&=u(-\xi \ul{B_j},0),\\
d\Sha(p(0,\underline{C_{k}});0)n(\gm)^{-1}&=u(-\xi \ul{C_k},-\xi \ul{C_k} \xi^*), \text{ and}\\
d\Sha(0;\underline{\gamma Z_{l}})n(\gm)^{-1}&=u\left(0,-\gm_G^2 \Xi(\gm Z) \ups^{-1}\right).\\
\end{split}
\end{equation*}
(Since $B_j \xi^*$ has image in $E$, we have $\xi B_j \xi^*=0$.)

Therefore (\ref{Guest}) is simply
\begin{equation*}
\begin{split}
(\Sha_S^* \omega_N)(1\times \gamma) [\mathfrak B] &= \omega_N(1)[\mathfrak B_2] \\
						 &= \omega_N(1)[L(\mathfrak B_1)] \\
 						&=(\det L(\gm))\omega_{N}(1)[\mathfrak B_1]\\
						&=\pm ( \det L(\gm)) \cdot \omega_{M/\Delta_S \times S}(1 \times \gamma)[\mathfrak B].\\
\end{split}
\end{equation*}						
 The result follows.
 \end{proof}
  
The fact that $\delta$ is nowhere vanishing has an important corollary. 
  
\begin{cor} \label{immersion} The map $\Sha_S$ is an immersion, and $N^S$ is an open subset of $N$.
\end{cor}
  
\begin{proof} This follows from Corollary \ref{dimension_formula}, Proposition \ref{delta/det}, and (\ref{trans_diagram}).
\end{proof}
     
\subsection{Choice of Differential Forms}

We now pin down differential forms on $M/\Delta_S \times S$ and $N$.  
The exact sequences from Section \ref{4X} give a natural way to build both of these forms from the same pieces.
      
Choose left-invariant differential forms $\omega_{G}$, $\omega_{H}$, $\omega_{S}$ and $\omega_{Z_{H}(S)}$ on the groups $G, H,S$, and $Z_H(S)$.
When convenient, we will simply write $\omega_G$ for $\omega_G(1)$ at the identity of $G$, and similarly for other groups.  Note that specifying an invariant differential form at $1$ prescribes its values on the entire group.
Also fix an alternating form $\omega_{(W',E)}$ of top degree on $\textrm{Hom}(W',E)$.
These five choices will determine all the forms that we want.

Write $\omega_{M}$ for the product of $\omega_G$ and $\omega_H$ on $M$.  
Next, using the exact sequence corresponding to $z_{\mathfrak h}(\mathfrak s)$, define $\omega_{H^{E^\circ}}$, a left invariant differential form on $H^{E^\circ}$, so that $\omega_{Z_{H}(S)}=\omega_{H^{E^\circ}}\cap\omega_{S}$. Using this and the exact sequence corresponding to $\kappa$, we define a top degree alternating form  $\omega_{\kappa}=\omega_{(W',E)}\cap\omega_{H^{E^\circ}}$.

We now define a form of top degree $\omega_{\mathfrak n_{\xi}}$ on $\mathfrak  n_{\xi}$ using the exact sequence corresponding to $\mathfrak m$, i.e., so that $\omega_{M}=\omega_{\mathfrak n_{\xi}}\cap\omega_{\kappa}$. 

Using the exact sequence corresponding to $\mathfrak n$ and the definition of $\omega_{\mathfrak n_{\xi}}$, we define $\omega_{N}$ as $\omega_{(W',E)}\cap \omega_{\mathfrak n_{\xi}}$.

Using the short exact sequence:
$$0 \to z_{\mathfrak h}(\mathfrak s)\overset{\Delta} \to \mathfrak m \overset{p} \to \mathfrak m/\Delta_\LS \to 0$$
(where $\Delta(z)=(\Xi(z),z)$) we define $\omega_{M/\Delta_S}$ such that it satisfies
$$\omega_{M}=\omega_{M/\Delta_S}\cap \omega_{Z_{H}(S)}.$$  
Also, denote by $\omega_{M/\Delta_S \times S}$ the form $\textrm{pr}_{1}^{*}(\omega_{M/\Delta_S})\wedge \textrm{pr}_{2}^{*}(\omega_{S})$ where $\textrm{pr}_{i}$ is the projection to the $i$th coordinate of $M/\Delta_S\times S$.
 
 Recall the choice of basis from the previous section.  The product
 \begin{equation*}
\omega_{M}[ (\underline{A_{i}},0),(0,\underline{B_{j}}),(\Xi(\underline{B'_{j'}}),\underline{B'_{j'}}),(0,\underline{C_{k}}),(\Xi(\underline{Z_{l}}),\underline{Z_{l}}) ]\omega_{S}(\underline{Z_{l}})
\end{equation*}
is equal, up to a sign, to
\begin{equation*}
 \omega_{M/\Delta_S\times S}(1,\gamma)[ (p(\underline{A_{i}},0);0);(p(0,\underline{B_{j}});0);(p(0,\underline{C_{k}});0);((0,0);\underline{\gamma Z_{l}}) ]\cdot {\omega_{Z_{H}(S)}(\underline{B'_{j'}},\underline{Z_{l}})}.
\end{equation*}

  Consider the basis $\mathfrak B_{1}$ defined above. Clearly, $\{u(\underline{A_{i}}\xi,0),u(0,\xi\underline{C_{k}}\xi^{*}),u(0,\xi\underline{Z_{l}}\xi^{*})\}$ is a basis of $\mathfrak n_{\xi}$ while $\{\psi(u(-\xi B_{j},0))\}$ is a basis of $\textrm{Hom}(W',E)$.  

We have
\begin{equation*}
\omega_{M}[\mathfrak B] = \pm \omega_{N}[\mathfrak B_{1}]\omega_{\hh^{E^\circ}}(\underline{B'_{j'}}),
\end{equation*}
and
 \begin{equation*}
 \omega_{M/\Delta_S\times S}(1,\gamma)[p \mathfrak B]=\pm \omega_{N}[\mathfrak B_{1}].
\end{equation*}

Remark:  We omit the calculation of the signs $\pm$ above, since it is only the associated measures that we require.  They depend only on the dimensions of the various groups.

\section{Root Vectors} \label{routes}

Our remaining task is to describe $L(\gm)$ and compute its determinant.
To this end, we demand that the basis $\mathfrak B$ of $\mathfrak m$ introduced in Section \ref{JacobianSection} respects certain root space decompositions.
We may pass to an algebraic extension of $F$ for this calculation, since this does not change the determinant.  Therefore the tori $S$ and $S_G$ are split, and we may take them to be diagonal.
In particular one may refer to the usual basis $e_i$ of characters of $S_G$.

\subsection{Root Vectors}

Let us record the image and fibres of the restriction map
\begin{equation*}
\res: R(\bG,\bS_{\bG}) \to \Hom({}^\xi \bS,\Gm),
\end{equation*}
which we write as $\alpha \mapsto \alpha_{\res}$.

Write $R(\bG,\bS_{\bG})^\theta$ for the fixed points of $R(\bG,\bS_{\bG})$ under $\theta$, and $R(\bG,\bS_{\bG})_0$ for the complement of $R(\bG,\bS_{\bG})^\theta$ in $R(\bG,\bS_{\bG})$.

\begin{lemma}
\begin{enumerate}
\item In the symplectic case, the map $\res$ maps $R(\bG,\bS_{\bG})$ onto $R(\bG^\theta,{}^\xi \bS)$.  Let $\beta \in R(\bG^\theta,{}^\xi \bS)$.  If $\beta$ is a long root, its fibre is a singleton in $R(\bG,\bS_{\bG})^\theta$.  If $\beta$ is a short root, its fibre consists of a $\theta$-orbit of roots in $R(\bG,\bS_{\bG})_0$.
\item In the orthogonal case, the map $\res$ maps $R(\bG,\bS_{\bG})_0$ onto $R(\bG^\theta,{}^\xi \bS)$.  If $\alpha \in R(\bG,\bS_{\bG})^\theta$, then $\alpha_{\res} \notin R(\bG^{\theta},{}^\xi \bS)$.  The fibres over $R(\bG^\theta,{}^\xi \bS)$ are  $\theta$-orbits of roots in $R(\bG,\bS_{\bG})_0$.
\end{enumerate}
\end{lemma}

Let $\alpha \in R(\bG,\bS_{\bG})$.  If $A_{\alpha}$ is a root vector for $\alpha$, then $\tau(A_{\alpha})$ is a root vector for $\theta(\alpha)$.  We will normalize our root vectors so that if $\{ \alpha, \theta(\alpha) \}$ is an orbit in $R(\bG,\bS_{\bG})_0$, then 
$A_{\theta(\alpha)}=\tau(A_{\alpha})$.  Note in the symplectic case that if $\alpha \in R(\bG,\bS_{\bG})^\theta$, then $\tau(A_{\alpha})=-A_{\alpha}$.
  
\begin{defn} Let $\alpha \in R(\bG,\bS_{\bG})_0$ and $\alpha' =\theta(\alpha)$.  Put $\beta=\alpha_{\res}$.  Put $A_{\beta}=X_\alpha-X_{\alpha'}$, and $C_\beta=\Xi^+(A_{\beta})$ (see Definition \ref {xidef}).
\end{defn}
Then $A_{\beta}$ is a root vector for $\beta$ in $\GG^{\theta}$, and $C_\beta$ is a root vector for $\Xi^+(\beta)$ in $\hh^E$.

Our transformation $L$ should preserve certain two and three-dimensional subspaces $\nn_\alpha$ of $\nn$.   
 
\begin{defn}
 Let $\alpha \in R(\bG,\bS_{\bG})_0$.  Write $\mathfrak n_{\alpha}$ for the three-dimensional space spanned by $u(A_{\alpha}\xi,0)$, $u(A_{\alpha'}\xi,0)$, and $u(0,\xi C_{\beta}\xi^{*})$. Define a linear transformation $L_{\alpha}:\mathfrak n_{\alpha}\to \mathfrak n$ so that :
 \begin{equation*}
 \begin{split}
L_{\alpha}( u(0,\xi C_{\beta}\xi^{*})) &= u(-\xi C_{\beta},-\xi C_{\beta}\xi^{*}) \\
L_{\alpha}(u(A_{\alpha}\xi,0)) &=u\big( A_{\alpha}\xi, A_{\alpha}\xi\xi^{*}+ \eta(\gamma)A_{\alpha}^{*} + A_{\alpha}\eta(\gamma) \big)\\
 L_{\alpha}(u(A_{\alpha'}\xi,0)) &=u\big( A_{\alpha'}\xi, A_{\alpha'}\xi\xi^{*}+ \eta(\gamma)A_{\alpha'}^{*} + A_{\alpha'}\eta(\gamma) \big).\\
\end{split}
\end{equation*}
\end{defn}

 If $\alpha$ corresponds to the root $e_{i}-e_{j}$ of $\bS_{\bG}$ in $\GL(\bW)$,  write $\lambda^{+}_{\alpha},\lambda^{-}_{\alpha}$ for the characters of $S_{G}$ corresponding to $e_{i}$ and $e_{j}$ respectively.   
\begin{prop} \label{nalpha}
$L_{\alpha}$ maps $\nn_{\alpha}$ to $\nn_{\alpha}$.  Viewing $L_{\alpha}$ as a linear endomorphism of $\nn_{\alpha}$, we have
\begin{equation*}
\det L_{\alpha}=\lambda^-_\alpha(\gm_G)-\lambda_\alpha^+(\gm_G).
\end{equation*}
 
\end{prop}

\begin{proof}
Let us check that the element 
\begin{equation}
A_{\alpha}\xi\xi^{*}+ \eta(\gamma)A_{\alpha}^{*} + A_{\alpha}\eta(\gamma)
\end{equation}
is a multiple of $\xi C_{\beta} \xi^*= (A_{\alpha}-A_{\alpha'}) \ups^{-1}$.
Multiplying on the right by $\ups$ gives
\begin{equation*}
\begin{split}
A_{\alpha}+ \gamma_G A_{\alpha'} + A_{\alpha}  \gamma_G &= A_{\alpha}+A_{\alpha'} \lambda_{\alpha'}^+(\gm_G)+A_{\alpha}\lambda_\alpha^-(\gm_G) \\
													&= A_{\alpha}(1+\lambda_{\alpha'}^+(\gm_G)+\lambda_{\alpha}^-( \gm_G))-\lambda_{\alpha'}^+( \gm_G) \Xi(C_\alpha).\\
\end{split}
\end{equation*}
It is not hard to see that generally 
\begin{equation*}
\lambda_\alpha^-(\tau(s))=\lambda_{\alpha'}^+(s)
\end{equation*}
for $s \in \bS_{\bG}$.
Together with the fact that $1+ \gm_G+\tau(\gm_G)=0$, this implies that $1+\lambda_{\alpha'}^+(\gm_G)+\lambda_{\alpha}^-( \gm_G)=0$.
Thus,
\begin{equation*}
L_{\alpha}(u(A_{\alpha}\xi,0))=u(A_{\alpha}\xi,0)-\lambda_{\alpha'}^+( \gm_G)u(0,\xi C_{\beta}\xi^{*}).
\end{equation*}
Similarly, 
\begin{equation*}
L_{\alpha}(u(A_{\alpha'}\xi,0))=u(A_{\alpha'} \xi,0) + \lambda_{\alpha}^+( \gm_G)u(0,\xi C_{\beta}\xi^{*}).
\end{equation*}
Of course,
\begin{equation*}
L_{\alpha}( u(0,\xi C_{\beta}\xi^{*})) =  -u(A_{\alpha}\xi,0)+u(A_{\alpha'}\xi,0)-u(0,\xi C_{\beta}\xi^{*}),
\end{equation*}
and the result follows.
\end{proof}

\begin{defn} Suppose we are in the symplectic case.  Let $\alpha \in R(\bG,\bS_{\bG})^\theta$. Fix a root vector $A_\alpha$ in $\GG$. Put $\beta=\alpha_{\res}$ and $C_\beta=\Xi^+(A_{\alpha})$.
  Write $\mathfrak n_{\alpha}$ for the two-dimensional space spanned by $u(A_{\alpha}\xi,0)$ and $u(0,\xi C_{\beta}\xi^{*})$. Define a linear transformation $L_{\alpha}:\mathfrak n_{\alpha}\to \mathfrak n$ so that :
 \begin{equation*}
 \begin{split}
L_{\alpha}( u(0,\xi C_{\beta}\xi^{*})) &= u(-\xi C_{\beta},-\xi C_{\beta}\xi^{*}) \\
L_{\alpha}(u(A_{\alpha}\xi,0)) &=u\big( A_{\alpha}\xi, A_{\alpha}\xi\xi^{*}+ \eta(\gamma)A_{\alpha}^{*} + A_{\alpha}\eta(\gamma) \big).\\
\end{split}
\end{equation*}
\end{defn}

Similarly to before we have:

\begin{prop} $L_{\alpha}$ maps $\nn_{\alpha}$ to $\nn_{\alpha}$.  Viewing $L_{\alpha}$ as a linear endomorphism of $\nn_{\alpha}$, we have
\begin{equation*}
\det L_{\alpha}=\lambda_\alpha^-(\gm_G)-\lambda_\alpha^+(\gm_G).
\end{equation*}
\end{prop}

{\bf Remark:} Suppose we are in the orthogonal case. Let $\alpha \in R(\bG,\bS_{\bG})^\theta$. Then similar calculations as in the proof of Proposition \ref {nalpha} gives that $$u\big( A_{\alpha}\xi, A_{\alpha}\xi\xi^{*}+ \eta(\gamma)A_{\alpha}^{*} + A_{\alpha}\eta(\gamma) \big)=u(A_{\alpha}\xi,0).$$

\subsection{$L$ and its determinant.}
Let us specify the choice of bases from Section \ref{JacobianSection} more precisely.  We will take the basis $\{ A_i\}$ of $\GG$ to be the union of a basis of $\LS_{\GG}$ and a basis of root vectors $A_\alpha$  for $S_G$, normalized as above.  Next, we choose for $\{ C_k \}$ the basis of root vectors of $S$ in $\hh^{E}$, with $C_\beta=\Xi(A_\beta)$ or $C_\alpha=\Xi(A_{\alpha})$ as specified in the previous section.
  
Let $\zeta:{}^\xi \LS \to {}^\xi \LS$ be the transformation given by 
\begin{equation*}
\zeta(Z)=\gm_G \tau(\gm_G)Z.
\end{equation*}
 
Note that $\det \zeta=\det \gm_G$.

$\newline$Now define $L:\mathfrak n\to \mathfrak n$ so that:
\begin{enumerate}
\item $L( u(A,0))=u(A,0)$ if $A\in \Hom(E,W)$. \\
\item $L( u(A\xi,0))=u\big( A\xi, A\xi\xi^{*}+ \eta(\gamma)A^{*} + A\eta(\gamma) \big)$ if $A\in \mathfrak s_{\mathfrak g}$. \\
\item $L|_{\mathfrak n_{\alpha}}=L_{\alpha}$ for $\alpha \in R(\bG,\bS_{\bG})$ in the symplectic case, and for $\alpha \in R(\bG,\bS_{\bG})_0$ in the orthogonal case. $L( u(A_{\alpha}\xi,0))=u(A_{\alpha}\xi,0)$ for $\alpha \in R(\bG,\bS_{\bG})^{\theta}$ in the orthogonal case.\\
\item $L( u(0,Z) )= u(0,\zeta (Z))$ if $Z\in {}^\xi \LS$.
\end{enumerate}
In fact, $L$ is constructed precisely to take the ordered basis $\mathfrak B_{1}$ to $\mathfrak B_{2}$ (see Section \ref{JacobianSection}).
 
\begin{prop} The quantity $|\det L|=|\det L(\gm)|$ is given by
\begin{equation} \label{bigprod}
| \det \gm_G| \cdot \prod_{\{\alpha \}}  |\lambda_\alpha^+(\gm_G)-\lambda_\alpha^-(\gm_G)|.
\end{equation}
In the orthogonal case, the product is taken over $\theta$-orbits $\{\alpha,\alpha'\}  \in R(\bG,\bS_{\bG})_0$, each of which has order $2$.  In the symplectic case, the product is taken over $\theta$-orbits in $R(\bG,\bS_{\bG})$, which have order $1$ or $2$.
\end{prop} 
 
\begin{lemma} Let $t \in S_G$.  Then 
\begin{enumerate}
\item  \begin{equation*}
\prod_{\alpha \in R(\bG,\bS_{\bG})_0}  \lambda_\alpha^+(t) \lambda_\alpha^-(t)=(\det t)^{2(\dim W-2)}.
\end{equation*}
\item In the symplectic case,
\begin{equation*}
\prod_{\alpha \in R(\bG, \bS_{\bG})^\theta}  \lambda_\alpha^+(t) \lambda_\alpha^-(t)=(\det t)^2.
\end{equation*}
\end{enumerate} 
\end{lemma}
 
Put 
\begin{equation*}
L_0(\gm)= \prod_{\{\alpha\} \in R_0}  |\lambda_\alpha^+(\gm_G)-\lambda_\alpha^-(\gm_G)|,
 \end{equation*}
where the product is over $\theta$-orbits $\{\alpha,\alpha'\}  \in R(\bG,\bS_{\bG})_0$.

Similarly set
\begin{equation*}
L_\theta(\gm)= \prod_{\{\alpha\} \in R^\theta}  |\lambda_\alpha^+(\gm_G)-\lambda_\alpha^-(\gm_G)|,
 \end{equation*}
where the product is over $\alpha  \in R(\bG,\bS_{\bG})^\theta$.

Thus $L(\gm)=|\det \gm_G| L_0(\gm) L_\theta(\gm)$ in the symplectic case, and $L(\gm)=|\det \gm_G| L_0(\gm)$ in the orthogonal case.

\begin{lemma}
\begin{enumerate}
 \item 
 \begin{equation*}
 \lambda_\alpha^+(\gm_G)-\lambda_\alpha^-(\gm_G)=\lambda_{\alpha'}^+(\gm_G)-\lambda_{\alpha'}^-(\gm_G)
 \end{equation*}
 \item 
 \begin{equation*}
  \lambda_\alpha^+(\gm_G)-\lambda_\alpha^-(\gm_G)=(\lambda_{\alpha}^-({}^\xi \gm)-\lambda_{\alpha}^+({}^\xi \gm))\lambda_{\alpha}^+(\gm_G) \lambda_{\alpha}^-(\gm_G) .
 \end{equation*}
 \end{enumerate}
 \end{lemma}
 
 We have
 \begin{equation*}
 \begin{split}
L_0(\gm)^2 &=  \prod_{\alpha \in R_0}  |\lambda_\alpha^+(\gm_G)-\lambda_\alpha^-(\gm_G)| \\
	&= |\det \gm_G|^{2(\dim W-2)} \prod_{\alpha \in R_0} |\lambda_{\alpha}^+({}^\xi \gm)-\lambda_{\alpha}^-({}^\xi \gm)| \\
		&=|\det \gm_G|^{2(\dim W-2)} \prod_{\alpha \in R_0} |\alpha({}^\xi \gm)-1|.\\
\end{split}
 \end{equation*}
 Similarly,
  \begin{equation*}
 \begin{split}
L_0(\gm) &=  \prod_{\alpha \in R^\theta}  |\lambda_\alpha^+(\gm_G)-\lambda_\alpha^-(\gm_G)| \\
	&= |\det \gm_G|^2 \prod_{\alpha \in R^\theta} |\lambda_{\alpha}^-({}^\xi \gm)-\lambda_{\alpha}^+({}^\xi \gm)| \\
		&=|\det \gm_G|^2 \prod_{\alpha \in R^\theta} |\alpha({}^\xi \gm)-1|.\\
\end{split}
 \end{equation*}
 
Orthogonal case: Regrouping gives
 \begin{equation*}
 \begin{split}
 L(\gm) &=|\det \gm_G|^{ \dim W-1} \left(\prod_{\alpha \in R_0} |\alpha({}^\xi \gm)-1| \right)^{\half} \\
 	&= |\det \gm_G|^{\dim W-1} |D_{H^E}(\gm)|.
\end{split} 
 \end{equation*}

Symplectic case: Regrouping gives
 \begin{equation*}
 \begin{split}
 L(\gm) &=|\det \gm_G|^{1+ \dim W} \left(\prod_{\alpha \in R_0} |\alpha({}^\xi \gm)-1| \right)^{\half} \left( \prod_{\alpha \in R^\theta} |\alpha({}^\xi \gm)-1| \right). \\
 	&= |\det \gm_G|^{\dim W+1} |D_{H^E}(\gm)|.
\end{split} 
 \end{equation*}

\begin{cor} Up to a sign, we have
\begin{equation*}
\delta(\gm)=(\det \gm_G)^{\dim W \pm 1}  D_{H^E}(\gm),
\end{equation*}
and 
\begin{equation*}
\Sha_S^*(\omega_N)=\delta_N(m) (\det \gm_G)^{\dim W \pm 1}  D_{H^E}(\gm) \omega_{M/\Delta_S} \wedge \omega_S.
\end{equation*}
Here $\pm 1$ is $+1$ in the symplectic case, and $-1$ in the orthogonal case.
\end{cor}

\section{The Integration Formulas} \label{LastSection}

\subsection{Haar measure}

Let $dn$ be a Haar measure on $N$.  Since $N^S$ is open in $N$ (by Corollary \ref{immersion}), we may restrict $dn$ to $N^S$.

By Proposition \ref{difffibre}, we obtain

\begin{prop} Let $f \in L^1(N^S,dn)$.  Then $(f\circ\Sha_{S}) \in L^1(M/\Delta_S \times S_{r}, \Sha_{S}^{*}(dn))$ and 
\begin{equation*}
\int_{N^S} f(n) dn=|W_H(S)|^{-1}  \int_{M/\Delta_S \times S_{r}}  (f\circ\Sha_{S})\Sha_{S}^{*}(dn).
\end{equation*}
\end{prop}

Recall the set $N_{r}=\bigcup_S N^S$. By Theorem \ref{Theorem 1}, the set of $\ol F$-points $\bN_{r}$ is a nonempty Zariski subset of the affine space $\bN$.
It follows that the set of $F$-points $N_{r}$ has negligible complement in $N$.  We obtain

\begin{prop} Let $f \in L^1(N,dn)$.  Then $f\circ\Sha_{S} \in L^1(M/\Delta_S \times S_{r}, \Sha_{S}^{*}(dn))$ for all $(\dim W)$-maximal tori $S$ of $H$, and
\begin{equation*}
\int_{N} f(n) dn=\sum_S |W_H(S)|^{-1}  \int_{M/\Delta_S \times S_{r}}  (f\circ\Sha_{S})\Sha_{S}^{*}(dn).
\end{equation*}
\end{prop}

The sum is taken over conjugacy classes of $(\dim W)$-maximal tori $S$ in $H$.  

By Proposition \ref{deltas}, our integration formula takes the form:

\begin{prop} Let $f \in L^1(N,dn)$ with $dn$ a Haar measure on $N$.  Then
\begin{equation*}
\int_N f(n) dn=\sum_S |W_H(S)|^{-1} \int_S |\delta(\gm)| \int_{M/\Delta_S} f(\Ad(m) n_S(\gm))|\delta_N(m)| \frac{dm}{dz} d \gm.
\end{equation*}
\end{prop}
Here $|\delta(\gm)|=|\det \gm_G|^{\dim W \pm 1} | D_{H^E}(\gm) |$.
Again, the sign $\pm$ is $+$ in the symplectic case, and $-$ in the orthogonal case.
  
\subsection{ $\Ad(M)$-invariant version} \label{Ad-invt}

The factor $\delta_N$ in the above formula suggests that we replace $dn$ with an $\Ad(M)$-invariant measure on $N$.   Such measures arise in the theory of intertwining operators (see \cite{S92}, \cite{GS98}).

\begin{prop} For $n=n(\xi,\eta) \in N'$, put $\phi(n(\xi,\eta))=|\delta_N(m(\eta \ups,1))|^{-\half}$.  Then 
\begin{enumerate}
\item For all $m \in M$ and $n \in N'$ we have $\phi(\Ad(m)(n))=\phi(n) |\delta_N(m)|^{-1}$.
\item $d_mn=\phi(n) dn$ is an $\Ad(M)$-invariant measure on $N$.
\end{enumerate}
\end{prop}

One computes in the orthogonal case that  
\begin{equation*}
\delta_N(m(g,h))=(\det g)^{\dim W+ \dim X -1},
\end{equation*}
and in the symplectic case that
\begin{equation*}
\delta_N(m(g,h))=(\det g)^{\dim W+ \dim X +1}.
\end{equation*}
(See Proposition 1 of \cite{IFS}.) 
 
Making this substitution for $dn$ we obtain
\begin{cor}
\begin{equation*}
\int_{N}f(n)d_{m}n=\sum_{S}|W_{H}(S)|^{-1}\int_{S} |\det L(\gm)| |\delta_N(\gm_G)|^{-\half}  \int_{M/S_{\Delta}}f(\Ad(m)n_{S}(\gamma))\frac{dm}{dz}d\gamma.
\end{equation*}
\end{cor}
Here we write $\delta_N(\gm_G)$ for $\delta_N(m(\gm_G,1))$.  Now we have
\begin{equation*}
\begin{split}
  |\det L(\gm)||\delta_N(\gm_G)|^{-\half} &=  |\det \gm_G|^{\pm \half+\half (\dim W-\dim X)}|D_{H^E}(\gm)|\\
  								&=|D_{H_E}(\gm)|^\half |D^\theta_G(\gm_G)|^\half |\det(\gm-1;E^\circ)|^{\half \dim E}, \\
  \end{split}
  \end{equation*}
  with the last equality following from Proposition 14 of \cite{IFS}.  To clarify, this is indeed a uniform formula for both the orthogonal and symmetric cases.
  
  Theorem \ref{intro_thm}, stated in the Introduction, follows from this.
  
\bibliographystyle{plain}

\end{document}